\documentclass[a4paper,11pt]{amsart}
\input{Macros}
\usepackage[frenchb, english]{babel}
\usepackage{tabu}

\newcommand{\fm}{\mathfrak{m}}
\newcommand{\fn}{\mathfrak{n}}
\newcommand{\fp}{\mathfrak{p}}

\DeclareMathOperator{\corank}{corank}
\DeclareMathOperator{\Quot}{Quot}

\DeclareMathOperator{\pr}{pr}

\title{On the singularities of quotients by 1-foliations}
\author{Quentin Posva}
\date{}

\address{Mathematisches Institut der Heinriche-Heine-Universität Düsseldorf, Universitätsstr.1, 24.13.03.33}
\email{quentin.posva@hhu.de}

\begin{document}

\maketitle

\begin{quote}
\textsc{Abstract.} We study the singularities of varieties obtained as infinitesimal quotients by $1$-foliations in positive characteristic. (1) We show that quotients by (log) canonical $1$-foliations preserve the (log) singularities of the MMP. (2) We prove that quotients by multiplicative derivations preserve many properties, amongst which most $F$-singularities. (3) We formulate a notion of families of $1$-foliations, and investigate the corresponding families of quotients. 
\end{quote}

\tableofcontents

\section{Introduction}
On a normal variety $X$ over a field of positive characteristic, a $1$-foliation is a saturated sub-sheaf of $T_X$ that is closed under Lie brackets and $p$-th powers. A fruitful construction, when having a $1$-foliation at our disposal, is the associated infinitesimal quotient of the underlying variety. It has been used multiple times to construct surprising or pathological examples in positive characteristic: see for example
\cite{Miyanishi_Russell_Purely_insep_cov_of_affine_plane, Kurke_Examples_of_false_ruled_surfaces,
Hirokado_Non_liftable_CY_3fold_in_char_3, Hirokado_Singularities_of_mult_derivations_and_Zariski_surfaces,  Hirokado_Zariski_surfaces_as_qt_of_Hirzebruch, Liedtke_Uniruled_surfaces_of_gen_type,
Roessler_Schroeer_Moret-Bailly_families}. The goal of this paper is to study the singularities of such quotients.

Since any finite purely inseparable morphism between normal varieties can be decomposed into a sequence of infinitesimal quotients, some restrictions on the $1$-foliations are necessary if we want qualitative results about the singularities. It is well-known that, at least on regular varieties, quotients by so-called multiplicative derivations are particularly simple \cite{Hirokado_Singularities_of_mult_derivations_and_Zariski_surfaces} (see \autoref{prop:normal_form_mult_derivations}). We give a mild generalization, in case the variety supporting the derivation is not necessarily regular. This applies in particular to quotient by $1$-foliations of rank $>1$ which are formally generated, up to saturation, by a finite set of commuting multiplicative derivations ($1$-foliations with \emph{at worst multiplicative singularities} in our terminology: see \autoref{def:at_worst_mult_sing}).

\begin{theorem_intro}[{Iterated applications of \autoref{thm:cohom_properties_of_mult_quotient}}]\label{thm_intro:qt_by_mult_foliations}
Let $X$ be a normal variety over a perfect field of characteristic $p>0$, and let $\sF$ be a $1$-foliation on $X$ with at worst multiplicative singularities. Then:
	\begin{enumerate}
		\item If $X$ is Cohen--Macaulay, so is $X/\sF$;
		\item If $X$ is $F$-pure (resp. $F$-rational, $F$-injective or $F$-regular), so is $X/\sF$.
	\end{enumerate}
\end{theorem_intro}

The idea is that any multiplicative derivation induces, after passing to an appropriate \'{e}tale neighbourhood, an action of the infinitesimal group scheme $\mu_p$. Quotients by $\mu_p$-actions have the remarkable property that the inclusion of the sub-ring of invariants is split (see eg \cite[I.2.11]{Jantzen_Representation_algebraic_groups}), and many cohomological properties are stable by taking a split sub-ring of equal dimension.

It is interesting to note that not every usual cohomological property descends to quotients by multiplicative derivations: for example, the Gorenstein property usually does not (see \autoref{rmk:toric_quotients_are_klt}), while the more geometric-flavoured $\bQ$-Gorenstein property does (\autoref{lemma:quotient_is_Q_Gor}). A detailed discussion of descent of properties such as Gorenstein and Buchsbaum (for the quotient by a single, non-necessarily $p$-closed derivation that is still ``multiplicative" in an appropriate sense) is given in \cite{Aramova_Avramov_Singularities_of_quotients_by_vector_fields} and \cite{Aramova_Reductive_derivations}.

\bigskip
Another perspective on singularities is given by the Minimal Model Program. One can define birational singularities of $\bQ$-Gorenstein $1$-foliations just as birational singularities of varieties and pairs, as first noticed by McQuillan \cite{McQuillan_Canonical_models_foliations}. A remarkable feature is that quotients by (log) canonical $1$-foliations can only \emph{improve} the birational singularities of the underlying variety:

\begin{theorem_intro}[The punchline of \autoref{thm:bir_sing_of_quotient}]\label{thm_intro:bir_sing_qt}
Let $X$ be a normal variety over a perfect field of characteristic $p>0$, and $\sF$ be a $1$-foliation on $X$.
	\begin{enumerate}
		\item If $\sF$ is canonical, then $X/\sF$ is terminal (resp. canonical, klt, log canonical) as soon as $X$ is.
		\item If $\sF$ is log canonical, then $X/\sF$ is klt (resp. log canonical) as soon as $X$ is.
	\end{enumerate}
\end{theorem_intro}

The proof is a direct comparison of pull-back formulas, leading to the simple equalities \autoref{eqn:discrepancies_of_quotient}. These equalities also show that any other implication (eg $X$ is midly singular if both $X/\sF$ and $\sF$ are) is not possible to obtain in general: see \autoref{rmk:bir_sing_qt} on that matter. We apply \autoref{thm_intro:bir_sing_qt} to normalized $p$-cyclic covering to obtain Reid-type covering results (ie description of the MMP singularities of cyclic covers), see \autoref{cor:Reid_lemma_insep_map}.

One drawback of \autoref{thm_intro:bir_sing_qt} is that it is difficult in general to decide whether a given $1$-foliation $\sF$ is log canonical, etc., along its singular locus. \autoref{prop:descent_along_grp_quotient} provides examples of log canonical $1$-foliations of corank $1$ on singular varieties. For $1$-foliations of rank $1$ on regular varieties there is a useful characterization of log canonicity due to McQuillan (\autoref{prop:lc_foliation_and_linear_alg}), which turns out to be equivalent to having at worst multiplicative singularities (\autoref{cor:lc_implies_mult_sing}). Using a local computation from \cite{Hara_Sawada_Splitting_of_Frobenius_sandwich}, this leads to a complete characterization of log canonical $1$-foliations on regular surfaces:
\begin{theorem_intro}[\autoref{lemma:reg_quotient_implies_reg_foliation}, \autoref{prop:can_foliations_on_surfaces} and \autoref{thm:lc_foliation_surfaces}]\label{thm_intro:lc_foliation_on_surfaces}
Let $S$ be a regular surface over a perfect field of characteristic $p>0$, and $\sF$ be a $1$-foliation of rank $1$ on $S$. Then:
	\begin{enumerate}
		\item $\sF$ is canonical if and only if $\sF$ is regular, if and only if $S/\sF$ is regular;
		\item $\sF$ is log canonical if and only if $S/\sF$ is $F$-regular, if and only if $S/\sF$ is $F$-pure.
	\end{enumerate}
(The conditions on $S/\sF$ can be equivalently be formulated for the normalization of $S$ in the field $K(S/\sF)^{1/p}$.) 
\end{theorem_intro}
This shows that the log canonical condition on $1$-foliations is quite restrictive, at least in dimension two. In higher dimensions, little more is known apart from the following:
\begin{corollary*}[\autoref{prop:can_foliations_on_surfaces} and \autoref{cor:lc_implies_mult_sing}]
Let $X$ be a regular variety over a perfect field of positive characteristic, and let $\sF$ be a non-trivial $1$-foliation.
	\begin{enumerate}
		\item $\sF$ cannot be terminal, and if it is canonical then it is regular outside a closed subset of codimension $\geq 3$.
		\item If $\sF$ has rank $1$, then it is canonical if and only if it is regular.
		\item If $\sF$ has rank $1$, then it is log canonical if and only if it has at worst multiplicative singularities.
\end{enumerate}	  
\end{corollary*}

Finally, we explore a notion of family of $1$-foliations. For technical reasons, we mostly restrict ourselves to the case of smooth fibrations $\sX\to S$ (see \autoref{def:families_of_foliations} and \autoref{rmk:technical_limitations}). In this setting, a family of $1$-foliations is a coherent sub-module $\sF\hookrightarrow T_{\sX/S}$ with $S$-flat cokernel, whose restriction to every fiber is a $1$-foliation in the usual sense. The theory of the Quot scheme implies that, when we consider trivial underlying fibration $\sX\to S$, universal families of $1$-foliations exist (\autoref{prop:universal_families_of_foliations}). Then we consider whether the operation of taking fibers commute with the operation of taking quotients by the $1$-foliation. In general they do not, but we prove the following criterion:
\begin{theorem_intro}[\autoref{prop:fiber_morphism_is_birational}]
Let $(\sX\to S,\ \sF\hookrightarrow T_{\sX/S})$ be a family of $1$-foliations. Assume that $\sX\to S$ is smooth and $S$ is regular. Let $\sQ=T_{\sX/S}/\sF$. Then if $\sQ_s$ is locally free for some $s\in S$, it holds that $\sX_s/\sF_s=(\sX/\sF)_s$.
\end{theorem_intro}

If $\sQ_s$ is just slightly less regular (for example not Cohen--Macaulay but $S_{\dim \sX_s-1}$), then commutativity may or may not occur, see \autoref{example:unexpected_commutativity} and \autoref{example:non-commutativity}. So a more refined criterion seems difficult to formulate in this generality. Nevertheless, a consequence of our definition of families of $1$-foliations is that the natural morphism $\varphi_s\colon \sX_s/\sF_s\to (\sX/\sF)_s$ is always an isomorphism in codimension one (\autoref{cor:commutativity_iff_S_2_fibers}), and this method easily produces $\bQ$-Gorenstein degenerations of regular schemes into non-$S_2$-ones.

%Finally, we also include a fun result (\autoref{thm:locally_alpha_2_quotient}) stating that, locally, any height one purely inseparable morphism between normal $\bF_2$-varieties is an $\alpha_2$-quotient. This should be compared to the similar statement about the ubiquity of $\alpha_\sL$-torsors obtained by Ekedahl \cite[Proposition 1.11]{Ekedahl_Canonical_models_surfaces_in_pos_char}. Ekedahl's result is global, but requires substantial regularity; our result is local, but only requires normality. The proof involves solving an algebraic differential equation.

\bigskip
The paper is organized as follows. In \autoref{section:preliminaries} we gather some basic material on derivations, foliations and infinitesimal quotients. Most of it is well-known and we claim no originality. However, the available material on $1$-foliations in positive characteristic is scattered through the literature (see in particular 
\cite{Miyaoka_Peternell_Geometry_of_higher_dim_varieties},
\cite{Patakfalvi_Waldron_Sing_of_general_fibers_and_LMMP},
\cite{Rudakov_Shafarevich_Inseparable_morphisms_of_alg_surfaces},
\cite{Liedtke_Uniruled_surfaces_of_gen_type},
\cite{McQuillan_Canonical_models_foliations},
\cite{Tziolas_Quotient_by_alpha_p_and_mu_p}
and \cite{Bernasconi_Counterexample_MMP_for_foliations_in_pos_char}), so we have chosen to reproduce it here for ease of reference. Large parts of the content of our \autoref{section:preliminaries} and \autoref{section:quotients_by_mult_der} are discussed extensively, with many examples and applications, in a book in preparation by Patakfalvi and Waldron \cite{Patakfalvi_Waldron_Pos_char_AG}. For the surface theory and extended examples, see also \cite{Tziolas_Topics_in_group_schemes_and_surfaces_in_pos_char}.

In \autoref{section:MMP_sing} we recall the definition of birational singularities for $1$-foliations, and establish a characterization of rank $1$ lc $1$-foliations. In \autoref{section:sing_qt} we investigate the singularities of some quotients, with the aim of proving \autoref{thm_intro:qt_by_mult_foliations}, \autoref{thm_intro:bir_sing_qt} and \autoref{thm_intro:lc_foliation_on_surfaces}. We discuss families of $1$-foliations in \autoref{section:variation}. 

\begin{remark}\label{rmk:non_p_closed_foliations}
We do not discuss properties of non-$p$-closed vector fields in this paper. Interesting features of those are given in \cite{Pereira_Invariant_hypersurfaces_for_pos_char_vector_fields, Mendson_Pereira_Codimension_one_foliations_in_pos_char}.
\end{remark}

\subsection{Acknowledgements}
I am grateful to Fabio Bernasconi, to Yuya Matsumoto and to the referee for many helpful comments. %, and to Przemys{\l}aw Grabowski for putting in perspective the content of \cite{Grabowski_Foliations_and_Galois_theorey_in_pos_char}. 
I would also like to thank Zsolt Patakfalvi, who introduced me to the topic of foliations in positive characteristic. The author is supported by the grant Postdoc Mobility $\sharp$P500PT/210980 of the Swiss National Science Foundation, and is thankful to Stefan Schröer's group at the Heinrich-Heine-Universität of Düsseldorf for hosting him.

\section{Preliminaries}\label{section:preliminaries}

\subsection{Notations}
Unless stated otherwise, we work over a perfect field $k$ of positive characteristic $p>0$. %(the main exception will be \autoref{section:lift_of_derivations}).
	\begin{enumerate}
		\item A \emph{variety} (over $k$) is an integral quasi-projective $k$-scheme of finite type. A \emph{curve} (resp. \emph{surface}, \emph{threefold}) is $k$-variety of dimension one (resp. two, three).
		\item Normalizations of integral Noetherian schemes and algebras are denoted by $(\bullet)^\nu$.
		\item Let $f\colon X\to S$ be a morphism of $k$-schemes. We let $F_\bullet$ denote the absolute Frobenius. Then we can form the \emph{relative Frobenius} $F_{X/S}=(F_X,f)\colon X\to X^{(1)}=X\otimes_{S,F_S}S$: it is an $S$-linear morphism. If $F_S$ is invertible, for example in the case $S=\Spec(k)$, we can construct the sequence of $S$-linear morphisms (all denoted by $F_{X/S}$),
				$$\dots\to X^{(-1)}\to X\to X^{(1)}\to \dots$$
		Notice that in this case, the schemes $X^{(n)}$ ($n\in \bZ$) are abstractly (but usually not $S$-linearly) isomorphic.
		\item The conditions $S_i$ are the \emph{Serre's conditions}, see \cite[0341]{Stacks_Project}.
		\item We use at several places standard MMP terminology for singularities, as defined in \cite[\S 2.3]{Kollar_Mori_Birational_geometry_of_algebraic_varieties}.
	\end{enumerate}

\subsection{$p$-basis}
Let $A$ be a $k$-algebra. 

\begin{definition}
Let $B\subset A$ be a purely inseparable extension of $k$-algebras, and assume that $A$ has height one over $B$ (that is: $A^p\subset B$). Then a \textbf{$p$-basis of $A$ over $B$} is a finite set of elements $\{a_1,\dots,a_n\}\subset A$ with the property that
		$$A=\bigoplus_{0\leq i_1,\dots,i_n<p}B\cdot a_1^{i_1}\cdots a_n^{i_n}$$
as $A^p$-modules.
\end{definition}

If $A$ is Noetherian, this notion is equivalent to that of \textbf{differential basis} \cite{Tyc_Diff_basis_p_basis}: a subset $\{a_1,\dots,a_n\}$ is a $p$-basis of $A$ over $B$ if and only if 
		$$\Omega^1_{A/B}=\bigoplus_{i=0}^n A\cdot d_{A/B}(a_i),$$
where $d_{A/B}\colon A\to \Omega^1_{A/B}$ is the universal derivation relative to $B\to A$. By Kunz theorem \cite[0EC0]{Stacks_Project}, if $A$ has a $p$-basis over $A^p$ then $A$ is regular and $F$-finite, and the converse also holds as $\Omega^1_{A/A^p}=\Omega^1_{A/k}$ will be a finite free $A$-module.

\begin{lemma}\label{lemma:p_basis_gives_sop}
Let $(A,\fm)$ be a regular complete local $k$-algebra, such that $k\subset A/\fm$ is a finite extension. Let $\{a_1,\dots,a_n\}$ be a $p$-basis of $A$ over $A^p$. Then:
	\begin{enumerate}
		\item for every $i$ we can write $a_i=\lambda_i+x_i$, where $\lambda_i\in A^\times$ and $x_i\in \fm\setminus\fm^2$;
		\item $\{x_1,\dots,x_n\}$ is a regular system of parameters of $A$.
	\end{enumerate}
\end{lemma}
\begin{proof}
As $k$ is perfect, we obtain that $A/\fm$ is also perfect. Moreover, by Cohen's structure theorem \cite[032A]{Stacks_Project}, $A$ contains a field of representatives $k_0$. So we can write $A=k_0\oplus\fm$ and $a_i=\lambda_i+x_i$ with $\lambda_i\in k_0$ and $x_i\in \fm$. By assumption the $d_{A/A^p}(a_i)$'s form a basis for $\Omega_{A/A^p}^1$ over $A$. Since $k_0=k_0^p$ and since $\Omega^1_{A/A^p}=\Omega^1_{A/k}$, we obtain that the $d_{A/k}(x_i)$'s form an $A$-basis of $\Omega^1_{A/k}$. In particular $n=\dim A$. Applying Nakayama's lemma to the isomorphism $\fm/\fm^2\cong \Omega^1_{A/k}\otimes A/\fm$, we obtain that $\fm=(x_1,\dots,x_n)$.
\end{proof}

\subsection{Derivations}
Let $R$ be a ring and $A$ be an $R$-algebra. A \textbf{derivation} of $A$ over $R$ is a $R$-linear map $D\colon A\to A$ satisfying the Leibniz rule
		$$D(ab)=aD(b)+bD(a),\quad a,b\in A.$$
The set of those, denoted $\Der_R(A)$, is naturally an $A$-module. This module is endowed with a Lie bracket 
		$$[\bullet,\bullet]\colon \Der_R(A)\to\Der_R(A)\quad [D,D']=D\circ D'-D'\circ D.$$
While the composition of two derivations might not be a derivation, in case $R$ is an $\bF_p$-algebra the $p$-fold composition affords an $R$-linear map
		$$\Der_R(A)\to\Der_R(A),\quad D\mapsto D^{[p]}=\underbrace{D\circ\dots\circ D}_{p\text{ times}}.$$			
Recall Hochschild's formula \cite[Theorem 25.5]{Matsumura_Commutative_Ring_Theory}: for $a\in A$ and $D\in\Der_R(A)$ we have
		\begin{equation}\label{eqn:Hochschild_formula}
		(aD)^{[p]}=a^pD^{[p]}+(aD)^{[p-1]}(a)D.
		\end{equation}
The $p$-th power of a sum of derivations is more complicated to describe: a formula of Jacobson \cite[p.209]{Jacobson_Abstract_derivations_and_Lie_algebras} shows that
		$$\left(\sum_i D_i\right)^{[p]}-\left(\sum_i D_i^{[p]}\right)
		\ \text{is a linear combination of multi-fold commutators in the } D_i\text{'s}.$$
In particular, the naive expression $(D_1+D_2)^{[p]}=D_1^{[p]}+D_2^{[p]}$ holds if $[D_1,D_2]=0$.

An alternative description of the module of derivation is given by the $A$-linear canonical isomorphism $\Hom_A(\Omega_{A/R}^1,A)\cong\Der_R(A)$, obtained by pre-composing any $\varphi\colon \Omega_{A/R}^1\to A$ by the universal $R$-linear derivation $d_{A/R}\colon A\to \Omega_{A/R}^1$. 

Given a multiplicatively closed subset $S\subset A$, there is a canonical map $\Der_R(A)\to\Der_R(S^{-1}A)$ given by the usual derivation rule for fractions. This is compatible with the localisation isomorphism $\Hom_A(\Omega^1_{A/R},A)\otimes S^{-1}A\cong \Hom_{S^{-1}A}(\Omega^1_{S^{-1}A/R},S^{-1}A)$. In particular the module of derivations sheafifies, and for any $R$-scheme $X$ we obtain a sheaf of $\sO_X$-module $\Der_R(\sO_X)$ which is the $\sO_X$-dual of $\Omega_{X/R}^1$. It is customary to write $\Der_R(\sO_X)=T_{X/R}$ (\footnote{In terms of the $T^i$ functors of Lichtenbaum and Schlessinger, we have $T_{X/R}=T^0_{X/R}$.}). The Lie bracket and $p$-fold composition also sheafify into $R$-linear operations on $\Der_R(\sO_X)$.

%\begin{remark}
%The presheaf $U\mapsto \Hom(\Omega_{X/k}^1|_U,\sO_U)$ is already a sheaf, so it gives a description of $\Der_k(\sO_X)$ on open subsets. However, since $\sO_U$ is in general not the sheafification of $\sO(U)$, and similarly for $\Omega_{X/k}^1|_U$, the inclusion $\Der_k(\sO(U))\hookrightarrow \Der_k(\sO_X)(U)$ is generally not an equality.
%\end{remark}

While the module of derivations commutes with localization, in general it does not commute with completion. Indeed, the module of Kähler differentials $\Omega_{\widehat{A}/R}$ of a complete local $R$-algebra $\widehat{A}$ is usually not of finite type over $\widehat{A}$. Still, we have the following result (which is known, but I could not locate a suitable compact reference):

\begin{lemma}\label{lemma:cont_der}
Let $(A,\fm)$ be a local ring essentially of finite type over a Noetherian ring $R$. Then there is a natural inclusion map
		$$\Der_R(A)\otimes \widehat{A}\hookrightarrow \Der_R(\widehat{A})$$
whose image is the sub-$\widehat{A}$-module $\Der^\text{cont}_R(\widehat{A})$ of continuous $R$-derivations of the $\fm$-adic completion $\widehat{A}$.
\end{lemma}
\begin{proof}
Since $\widehat{A}$ is a flat $A$-module and $\Omega_{A/R}^1$ is a finitely presented $A$-module, the canonical morphism
	$$\Der_R(A)\otimes_A\widehat{A}\longrightarrow \Hom_{\widehat{A}}(\Omega_{A/R}^1\otimes_A \widehat{A},\widehat{A})$$
is an isomorphism \cite[Chapter I, \S 2, n.10, Proposition 11]{Bourbaki_AC_I-II}. In the rest of the proof, we describe the target of this isomorphism. By the universal property of the inverse limit, it can be written as
		$$\Hom_{\widehat{A}}(\Omega_{A/R}^1\otimes \widehat{A},\widehat{A})=\varprojlim \Hom_{\widehat{A}}(\Omega_{A/R}^1\otimes \widehat{A},A/\fm^n).$$
Let us describe the Hom-module into $A/\fm^n$. On the one hand, an $\widehat{A}$-linear morphism $\Omega_{A/R}^1\otimes \widehat{A}\to A/\fm^n$ is always continuous for the natural topologies, since it is uniquely specified by an $A$-linear map $\Omega^1_{A/R}\to A/\fm^n$. On the other hand by \cite[6.Exercise 1.3]{Liu_AG_and_arithmetic_curves} we have a canonical isomorphism
		$$\Omega^1_{A/R}\otimes \widehat{A}\cong \varprojlim \left(\Omega^1_{\widehat{A}/R}/\fm^n\Omega^1_{\widehat{A}/R}\right).$$
Combining these two facts with \cite[20.7.14.4]{EGA_IV.1}, which we can apply as $A/\fm^n$ is discrete and is annihilated by $\fm^n$, we obtain a canonical identification
		$$\Hom_{\widehat{A}}(\Omega^1_{A/R}\otimes \widehat{A},A/\fm^n)\cong \Hom^{\text{cont}}_{\widehat{A}}(\Omega^1_{\widehat{A}/R}, A/\fm^n).$$
Let us apply the inverse limit along $n$: by the token already used above, it amounts to the same to apply the inverse limit on the second arguments of the Hom modules, and so we get an isomorphism
		$$\Hom_{\widehat{A}}(\Omega^1_{A/R}\otimes \widehat{A},\widehat{A})\cong \Hom^{\text{cont}}_{\widehat{A}}(\Omega^1_{\widehat{A}/R}, \widehat{A}).$$
The right-hand side is a sub-module of $\Hom_{\widehat{A}}(\Omega_{\widehat{A}/R}^1,\widehat{A})=\Der_R(\widehat{A})$ which, by \cite[20.4.8.2]{EGA_IV.1}, corresponds to the set of \emph{continuous} $R$-derivations of $\widehat{A}$ into itself. This completes the proof.
\end{proof}

In any case, we will use the following convention:
\begin{terminology}\label{terminology:formally_generated}
Assume that $A$ is a local ring. If $M\subset \Der_R(A)$ is a sub-module, then we will say that $M$ satisfies some property \emph{formally}, if the sub-$\widehat{A}$-submodule $M\otimes\widehat{A}\subset \Der_R(A)\otimes\widehat{A}$ satisfies the said property. %If $M$ is finitely generated, we will write $\widehat{M}=M\otimes\widehat{A}$ as it is customary.
\end{terminology}

\begin{remark}\label{lemma:descent_derivation_from_completion}
Let $(A,\fm)$ be a regular local ring essentially of finite type over a perfect field $k_0$. Assume that $A/\fm=k_0$. If $x_1,\dots,x_n$ is a regular system of parameters, then:
	\begin{enumerate}
		\item $\Der_{k_0}(A)$ is freely generated by some derivations $D_1,\dots,D_n$ such that $D_i(x_j)=\delta_{ij}$ (the Kronecker delta). This follows from that the $dx_i$ give a basis of $\Omega_{A/{k_0}}^1$ (see eg \cite[II.8.7-8]{Hartshorne_Algebraic_Geometry}).
		\item Under the isomorphism $\widehat{A}\cong k_0\llbracket x_1,\dots,x_n\rrbracket$, the $\widehat{A}$-module $\Der^{\text{cont}}_{k_0}(\widehat{A})\cong \Der_{k_0}(A)\otimes \widehat{A}$ is freely generated by the continuous $k_0$-derivations $\frac{\partial}{\partial x_i}=D_i\otimes 1$ $(i=1,\dots,n)$. This follows from the previous item and from \autoref{lemma:cont_der}.
	\end{enumerate}
\end{remark}

\subsubsection{$p$-closed, additive and multiplicative derivations}
Let $R$ be an $\bF_p$-algebra and $A$ be an $R$-algebra. We say that $D\in \Der_R(A)$ is \textbf{$p$-closed} if there is $a\in A$ such that $D^{[p]}=aD$. Hochschild's formula \autoref{eqn:Hochschild_formula} shows that any scaling of a $p$-closed derivation is still $p$-closed. Amongst $p$-closed derivations, we distinguish two special types as follows.

\begin{definition}
We say that $D$ is \textbf{additive} if $D^{[p]}=0$. We say that $D$ is \textbf{multiplicative} if $D^{[p]}=uD$ for some unit $u\in A^\times$.
\end{definition}

\begin{examples}\label{example:toric_derivation}
(Recall that $k$ stands for a perfect field of characteristic $p>0$.)
\begin{enumerate}
	\item The derivation $x^i\frac{\partial}{\partial x}$ on $k[x,y_1,\dots,y_{n}]$ is additive for $i\neq 1$, and multiplicative for $i=1$.
	\item Consider the derivation $\partial_{a,b}=ax\frac{\partial}{\partial x}+by\frac{\partial}{\partial y}$ on $k[x,y]$, where $a,b\in \bF_p$. We have
	$$\left(\partial_{a,b}\right)^{[p]}(x^iy^j)=(ai+bj)^px^iy^j.$$
The element $ai+bj$ is to be understood as an element of $\bF_p$, on which the Frobenius is trivial. Thus we see that $(\partial_{a,b})^{[p]}=\partial_{a,b}$. So $\partial_{a,b}$ is $p$-closed and multiplicative.
	\item Consider the derivation $D=x\frac{\partial}{\partial y}+y\frac{\partial}{\partial x}$ on $k[x,y]$. For $p=2$ we have $D^{[2]}=x\frac{\partial}{\partial x}+y\frac{\partial}{\partial y}$ which is not a scaling of $D$, so $D$ is not $p$-closed. For $p\neq 2$, the coordinate change $x=u+v, y=u-v$ gives $D=u\frac{\partial}{\partial u}+v\frac{\partial}{\partial v}$ which is $p$-closed and multiplicative.
\end{enumerate}
\end{examples}

\begin{warning}\label{warning}
The additive and multiplicative properties are usually not stable by scaling. This can be seen from Hochschild's formula \autoref{eqn:Hochschild_formula}.
Actually a scaling of an additive derivation can become multiplicative, and vice-versa, for instance $\frac{\partial }{\partial x}$ and $x\frac{\partial}{\partial x}$ on $k[x, x^{-1}]$. 
\end{warning}

There is a well-known correspondence between actions of $\alpha_p$ (resp.\ of $\mu_p$) on $X$, and derivations $D\in \Gamma(X,\Der_k(\sO_X))$ satisfying $D^{[p]}=0$ (resp.\ $D^{[p]}=D$). This correspondence is in fact slightly more general, as we explain now following \cite[\S 2]{Schroeer_Enriques_surfaces_with_K3_like_cover}.
For $\lambda\in k$ we construct an affine $k$-group scheme $G_\lambda$ as follows. Define a comultiplication on $k[x]/(x^p-\lambda x)$ via
		$$\Delta(u_i)=\sum_{j+l=i}u_j\otimes u_l, \quad
		u_s=x^s/s! \text{ for }0\leq s\leq p-1.$$
This gives a ring structure to the dual $\sO(G_\lambda)=\Hom_k(k[x]/(x^p-\lambda x),k)$. This defines $G_\lambda=\Spec \sO(G_\lambda)$ as a $k$-scheme; the group scheme structure is described in \cite[II, \S 7, Proposition 3.9]{Demazure_Gabriel_Groupes_algebriques}. As special cases, we have $G_0=\alpha_p$ and $G_1=\mu_p$, and one can show that $G_\lambda\cong G_1$ when $\lambda\neq 0$.

\begin{proposition}\label{prop:derivations_and_group_actions}
Let $X$ be a $k$-scheme and $D\in\Gamma(X,\Der_k(\sO_X))$ be a derivation. Then that $D^{[p]}=\lambda D$ for $\lambda\in k$ if and only if there is a $G_\lambda$-action on $X$ given by
		$$\sO_X\to \sO_X\otimes_k \sO(G_\lambda), \quad
		s\mapsto \sum_{i=0}^{p-1}\frac{D^{\circ i}(s)}{i!}\otimes t^i
		$$
where $t\in \sO(G_\lambda)$ is the dual section of $u_1$.
\end{proposition}
\begin{proof}
This is described in \cite[pp.10-11]{Schroeer_Enriques_surfaces_with_K3_like_cover}.
\end{proof}

\begin{remark}
The case $D^{[p]}=uD$ where $u\notin k$ does not correspond to a $k$-group action on $X$. However, as we will see in \autoref{section:quotients_by_mult_der}, if $u\in A^\times$ we recover a $k$-group action after a finite \'{e}tale cover.
\end{remark}

\begin{example}
Let $E$ be an elliptic curve over $k$, and let $\eta\in H^0(E,T_{E/k})$ be a global generator. Then 
		$$\eta^{[p]}=
		\begin{cases}
		\eta & \text{if }E\text{ is ordinary,}\\
		0 & \text{if }E\text{ is supersingular,}
		\end{cases}$$
see e.g. \cite[12.4.1.3]{Katz_Mazur_Arithmetic_elliptic_curves}. So $E$ is ordinary (resp. supersingular) if and only if $\mu_p$ (resp. $\alpha_p$) acts non-trivially on $E$.
\end{example}

\begin{lemma}\label{lemma:fixed_locus_of_action}
Let $G_\lambda$ acts on a $k$-scheme $X$ by means of a derivation $D$. Then the ideal of the fixed locus of the action is the ideal generated by $D(\sO_X)$, and the action is free outside the fixed locus.
\end{lemma}
\begin{proof}
We may assume that $X=\Spec(A)$ is affine. As $G_\lambda$ has no non-trivial subgroup schemes, the action is free outside the fixed locus. The action can be described as follows: given morphisms of affine schemes $f\colon (S'\to S\to X)$, corresponding to ring maps 
	$$f^*\colon \left(A\overset{\varphi}{\longrightarrow} \Gamma(S,\sO_S)\overset{\psi}{\longrightarrow}\Gamma(S',\sO_{S'})\right)$$ 
and given $\nu\in G_\lambda(S)$, the morphism $\nu\cdot f\colon S'\to X$ corresponds to the composition
	$$\begin{tikzcd}
	A\arrow[r] &
	A\otimes_k \sO(G_\lambda) \arrow[r, "\varphi^*\otimes \nu^*"] &
	\Gamma(S,\sO_S)\otimes \Gamma(S,\sO_S) \arrow[r, "\psi^*\otimes \psi^*"] &
	\Gamma(S',\sO_{S'})\otimes \Gamma(S',\sO_{S'}) \arrow[d, "m"] \\
	&&& \Gamma(S',\sO_{S'})
	\end{tikzcd}$$
where the first arrow corresponds to the action morphism, and $m(x\otimes y)=xy$. In particular,
		$$(\nu\cdot f)^*(a)=
		\sum_{i=0}^{p-1}
		\frac{f^*D^{\circ i}(a)}{i!}\nu^*t^i$$
(with the convention that $0^0=1=0!$ and $D^0=\id$). %Similarly, if we have instead an $\mu_p$-action, then given $f$ as above and $\nu\in \mu_p(S)=\{u\in \Gamma(S,\sO_S)\mid u^p=1\}$, the morphism $\nu\cdot f$ is given by the ring map
%		$$A\to \Gamma(S',\sO_{S'}),\quad
%		a\mapsto \sum_{i=0}^{p-1}\frac{f^*(D^{\circ i}(a))}{i!}\psi(\nu^i-1).$$
The fixed locus of the $G_\lambda$-action is the subscheme $Z\subset X$ with the following property (see eg \cite[2.2.5]{Brion_Some_structure_thms_for_alg_groups}): $S\to X$ factors through $Z$ if and only if $\nu\cdot f=f$ for every $f$ and $\nu$. From the description above it is then clear that $Z$ is the closed subscheme of $X$ whose ideal is generated by the set $\{D(a)\mid a\in A\}$.
\end{proof}

\begin{remark}
\autoref{prop:derivations_and_group_actions} can be generalized as follows.
	\begin{enumerate}
		\item Given integers $n,m\geq 0$, actions of $\mu_p^{\times n}\times\alpha_p^{\times m}$ on $X$ corresponds bijectively to sets of $n+m$ derivations $\{D_1,\dots,D_{n+m}\}\subset \Der_k(\sO_X)(X)$ such that 
	\begin{itemize}
		\item $D_i^{[p]}=D_i$ for $1\leq i\leq n$,
		\item $D_j^{[p]}=0$ for $n+1\leq j\leq n+m$, and
		\item $D_a\circ D_{b}=D_b\circ D_a$ for any $1\leq a,b\leq n+m$.
	\end{itemize}
		\item Let $G$ be either $\alpha_{p^n}$ or $\mu_{p^n}$. Disregarding the Hopf algebra structure of $\sO(G)$, we have $\sO(G)\cong k[t]/(t^{p^n})$. Thus a scheme morphism $a\colon G\times X\to X$ such that 
			$$\left(X\cong \{e_G\}\times X\hookrightarrow G\times X\overset{a}{\longrightarrow} X\right)=\id_X$$
		corresponds to a $k$-linear ring map $a^*\colon \sO_X\to\sO_X\llbracket t\rrbracket /(t^{p^n})$ which reduces to $\id_{\sO_X}$ modulo $t$. By \cite[\S 27]{Matsumura_Commutative_Ring_Theory} such morphisms correspond bijectively to Hasse--Schmidt derivations $\bold{D}$ of length  $p^n$. Unravelling the compatibility conditions that are necessary for $a$ to be an action, we obtain necessary and sufficient conditions for $\bold{D}$ to define a $G$-action. For example, if $G=\alpha_{p^n}$ then we obtain that $\bold{D}$ is iterative. Actions of finite products of $\alpha_{p^n}$ and $\mu_{p^m}$ are then described in the same way as above.
	\end{enumerate}
\end{remark}

In general, a given $p$-closed derivation $D\in\Der_{\bF_p}(A)$ is neither additive not multiplicative. So one can ask whether there is $0\neq a\in A$ such that $aD$ becomes additive or multiplicative.

It is always possible to scale $A$ so that it becomes additive: the following argument was kindly communicated to me by Yuya Matsumoto.
\begin{lemma}\label{lemma:scaling_is_additive}
Let $A$ be an integral $\bF_p$-algebra, and $0\neq D\in\Der_{\bF_p}(A)$ be a $p$-closed derivation. Then there exists $a\in A$ such that $aD$ is non-zero and additive.
\end{lemma}
\begin{proof}
Choose $x\in A$ such that $D(x)\neq 0$ and write $a=D(x)^{p-1}$. I claim that $aD$ is additive. By Hochschild's formula \autoref{eqn:Hochschild_formula} the derivation $aD$ is $p$-closed, say $(aD)^{[p]}=h\cdot aD$. Then
		$$(ha D)(x)=(aD)^{[p]}(x)
		=(aD)^{\circ (p-1)}(a D(x))
		=(aD)^{\circ (p-1)}(D(x)^p)=0.$$
Since $A$ is a domain and $(aD)(x)=D(x)^p\neq 0$, we deduce that $h=0$.
\end{proof}

On the other hand, we observe the following:
\begin{remark} 
It is not always possible to scale $D$ so it becomes multiplicative: indeed, if that was the case, then $A^D$ would be a multiplicative quotient and thus, assuming that $A$ is Cohen--Macaulay, we would obtain that $A^D$ is also Cohen--Macaulay by \autoref{thm:cohom_properties_of_mult_quotient}. This is usually not the case: for example, using \autoref{rmk:additive_quotients_are_only_S2} one sees that for
		$$D=x^p\frac{\partial}{\partial x}+y^p\frac{\partial}{\partial y}+z^p\frac{\partial}{\partial z}\quad \text{on }A=k[x,y,z],$$
the invariant subring $A^D$ is three-dimensional but only $S_2$. So no non-zero scaling of $D$ can be multiplicative.
\end{remark}

\subsection{Foliations}
Let $X$ be a normal connected $k$-scheme of finite type. 

\begin{definition}
A \textbf{foliation} is a coherent subsheaf $\sF\subset T_{X/k}$ which is saturated in $T_{X/k}$ (ie the quotient $T_{X/k}/\sF$ is a torsion-free $\sO_X$-module) and closed under Lie brackets. A foliation is called a $\mathbf{1}$\textbf{-foliation} if it is also closed under $p$-th powers (\footnote{
		The terminology is not consistent in the literature. What we call $1$-foliations are called \emph{$p$-foliations}, or sometimes simply \emph{foliations}, in other sources. Foliations which are not necessarily closed under $p$-th powers have not been studied very extensively: see \autoref{rmk:non_p_closed_foliations}.
}).
\end{definition}

%\begin{remark}
%Although it is tempting to call a foliation additive (resp. multiplicative) if it is locally generated, up to saturation, by an additive (resp. multiplicative) derivation, we refrain of doing so because of \autoref{warning}.
%\end{remark}

The geometric significance and relevance of $1$-foliations is made clear by Jacobson's correspondence, which we will state in the next subsection (\autoref{thm:Jacobson_correspondence}). Of course, $T_{X/k}$ and the zero sub-sheaf are $1$-foliations, which we refer to as the trivial ones. 

\begin{remark}\label{rmk:Lie_and_pth_cond_at_generic_point}
Taking the stalk at the generic point establishes a bijective correspondence between saturated coherent subsheaves of $T_{X/k}$ and sub-$k(X)$-vector spaces of $T_{k(X)/k}$. Closure under Lie brackets or $p$-th powers are also properties determined at the generic point. In particular, given a coherent subsheaf of $T_{X/k}$ which is generically closed under Lie brackets (resp. under Lie brackets and $p$-th powers), its saturation in $T_{X/k}$ yields a foliation (resp. a $1$-foliation).
\end{remark}

\begin{definition}
Let $\sF$ be a foliation on $X$ and let $\eta\in X$ be the unique generic point. The \textbf{rank} of $\sF$ is $\dim_{k(\eta)}\sF_\eta$, and the \textbf{corank} of $\sF$ is $\dim_{k(\eta)}(T_{X/k}/\sF)_\eta$. We have the relation $\rank(\sF)+\corank(\sF)=\dim X$.
\end{definition}

\begin{definition}\label{def:reg_foliation}
Let $\sF$ be a foliation on $X$, and let $x\in X$ be a point. We say that $\sF$ is \textbf{regular} at $x$ if $\sO_{X,x}$ is regular and the $\sO_X$-module $T_{X/k}/\sF$ is locally free at $x$. Otherwise, $\sF$ is \emph{singular} at $x$.
\end{definition}

We note that the singular set of a foliation is closed, of codimension $\geq 2$, and contains the singular locus of the underlying variety. Notice also that $\sF$, as $\sO_X$-module, is reflexive (equivalently it satisfies Serre's condition $S_2$) by \cite[0EBG]{Stacks_Project}.

Regular $1$-foliations have a simple local description on regular varieties:
\begin{lemma}[Seshadri, Yuan]\label{example:reg_foliations_on_reg_varieties}
Let $(A,\fm)$ be a regular local algebra that is essentially of finite type over $k$ and such that $A/\fm=k$, and $\sF\subset \Der_k(A)$ be a regular $1$-foliation on $A$. Then we can find local coordinates $x_1,\dots,x_n$ of $A$ such that 
		$$\sF=\sum_{i=1}^rA\cdot D_i, \quad r=\rank \sF,$$
where the $D_i$ are as in \autoref{lemma:descent_derivation_from_completion}. In particular, $A$ has a $p$-basis over $A^\sF$.
\end{lemma}
\begin{proof}
See \cite[Proposition 6]{Seshadri_Operateur_de_Cartier} or \cite[Proof of Theorem 12]{Yuan_Inseparable_Galois_theory}. %(The fact that $\sF$ is closed under $p$-th roots is necessary: consider the sub-module of derivations generated by $\partial_y+x\partial_x$ on $(\bold{0}\in\bA^2)$.)
\end{proof}

More generally, regular (non-necessarily $p$-closed) derivations of regular \emph{complete} local ring also admit a normal form, see \cite[Divertimento II.1.6]{McQuillan_Canonical_models_foliations}, but we will not need such a description.

For invertible foliations on regular schemes, there is a simple characterisation of freeness.

\begin{lemma}\label{lemma:singularity_Gor_foliation}
Let $x\in X$ be a regular point and assume that the foliation $\sF\otimes \sO_{X,x}$ is invertible as $\sO_{X,x}$-module. Then $\sF$ is regular at $x$ if and only if $\sF\not\subset \fm_{X,x}T_{X/k}$.
\end{lemma}
\begin{proof}
We apply \cite[II.8.9]{Hartshorne_Algebraic_Geometry} to the cokernel $\sN$ of $\sF\hookrightarrow T_{X/k}$: it is free at $x$ if and only if $\rank(\sN\otimes k(x))=\dim \sO_{X,x}-1$. This is equivalent to left-exactness of the right-exact sequence
		$$0\to \sF\otimes k(x)\to T_{X/k}\otimes k(x)\to \sN\otimes k(x)\to 0.$$
If $\sF$ is generated by $\partial$ at $x$, then left-exactness holds if and only if $\partial\notin \fm_{X,x}T_{X/k}$.
\end{proof}

Next we generalize the regularity condition in the following way (recall \autoref{terminology:formally_generated} about formal properties of foliations).

\begin{definition}\label{def:at_worst_mult_sing}
Notations as above. We say that $\sF$ has \textbf{at worst multiplicative singularities} if at every point it is generated formally and up to saturation 
by multiplicative continuous derivations that commute with each other.
\end{definition}

\begin{examples}\label{rmk:at_worst_mult_sing}
	\begin{enumerate}
		\item Regular $1$-foliations have at worst multiplicative singularities. Indeed, by working formally we reduce through \autoref{example:reg_foliations_on_reg_varieties} to $X=\bA^n_{\bold{x}}$ and $\sF=\sum_{i=1}^r \sO_X\cdot\frac{\partial}{\partial x_i}$. Then while the $\frac{\partial}{\partial x_i}$ are additive, $\sF$ is the saturation of $\sum_{i=1}^r \sO_X\cdot x_i\frac{\partial}{\partial x_i}$, which is generated by multiplicative derivations commuting with each other. 
		\item Let $X=\bA^3_{\bold{x}}$ and $\sG=\sO_X\cdot \left(x_1\frac{\partial}{\partial x_1}+x_2\frac{\partial}{\partial x_2}\right)+\sO_X\cdot \frac{\partial}{\partial x_3}$. Then $\sG$ is the saturation of the sub-module generated by the multiplicative derivations $x_1\frac{\partial}{\partial x_1}+x_2\frac{\partial}{\partial x_2}$ and $x_3\frac{\partial}{\partial x_3}$, and these two derivations commute with each other. So $\sG$ has at worst multiplicative singularities.
	\end{enumerate}
\end{examples}

\begin{remark}\label{rmk:local_descript_mult_sing}
Much like regular $1$-foliations, $1$-foliations with at worst multiplicative singularities have simple \emph{formal} local descriptions on regular varieties. It will follow from \autoref{prop:normal_form_mult_derivations} that if $A$ is a regular local $k$-algebra and $\sF\subset \Der_k(A)$ is a $1$-foliation with at worst multiplicative singularities, then we can find formal local coordinates $x_1,\dots,x_n$ of $\widehat{A}$ such that 
	$$\widehat{\sF}=\sum_\alpha \widehat{A}\cdot D_\alpha,\quad\text{where}\quad D_\alpha= \sum_{i=1}^n\lambda_{\alpha i}x_i\frac{\partial}{\partial x_i}\quad \text{with }\lambda_{\alpha i}\in\bF_p.$$
The assumption that $\sF$ is generated formally by \emph{commuting} derivations is needed in order to find formal coordinates adapted to every $D_\alpha$.
\end{remark}

\begin{example}\label{example:sing_of_toric_derivation}
Consider the derivation $\partial_{a,b}$ on $\bA^2$ introduced in \autoref{example:toric_derivation}, and let $\sF_{a,b}$ be the saturation of $\sO\cdot\partial_{a,b}$. The sheaf $\sF_{a,b}$ is closed under Lie brackets, essentially because it is generated up to saturation by a single derivation. We have also seen in \autoref{example:toric_derivation} that $\partial_{a,b}$ is $p$-closed, so $\sF_{a,b}$ is a $1$-foliation. Notice that $\sF_{a,b}=\sF_{\lambda a,\lambda b}$ for any $\lambda\in \bF_p^\times$. Let us look at its singularities.
	\begin{enumerate}
		\item If $ab=0$ then $\sF_{a,b}$ is generated by either $\frac{\partial}{\partial x}$ or $\frac{\partial}{\partial x}$, and hence it is regular everywhere.
		\item If $ab\neq 0$ then $\sF_{a,b}$ is generated by $\partial_{a,b}$ and has a unique singularity at the origin.
	\end{enumerate}
Hence the $1$-foliation $\sF_{a,b}$ has at worst multiplicative singularities.
\end{example}

\begin{construction}[Birational pullback]\label{construction:birational:pullback}
Let $\pi\colon Y\to X$ be a birational morphism of normal connected $k$-schemes. The generic stalk $\sF_{k(X)}$ determines a foliation on $Y$, which we will usually denote by $\pi^*\sF$. If $\sF$ is a $1$-foliation, then so is $f^*\sF$ by \autoref{rmk:Lie_and_pth_cond_at_generic_point}.
\end{construction}

\begin{example}\label{example:smth_blow_up}
The reader will check that if $\pi\colon \bA^n_{\bold{y}}\to \bA^n_{\bold{x}}$ is the $y_1$-chart of the blow-up of $(x_1,\dots,x_r)$ for $r\leq n$, which means that we have
				$$(x_1,\dots,x_n)\mapsto 
				(y_1,y_1y_2,\dots,y_1y_r,y_{r+1},\dots,y_n),$$
		then the transformation rules are
				$$\pi^*\partial_{y_1}=\partial_{y_1}-\sum_{i=2}^r\frac{y_i}{y_1}\partial_{y_i}, \quad
				\pi^*\partial_{x_i}=\frac{1}{y_1}\partial_{y_i} \ (1<i\leq r), \quad
				\pi^*\partial_{x_j}=\partial_{y_j} \ (j>r).$$
\end{example}

\begin{definition}
Let $\sF$ be a foliation on $X$. A prime divisor $E\subset X$ is called \textbf{invariant} for $\sF$ if generically\footnote{This definition is sometimes stated without the genericity assumption. To avoid technicalities related to $E\hookrightarrow X$ not being a regular immersion at special points, we state it as a condition at the generic point of $E$. As long as $E$ is a divisor, this makes no difference in the proofs.} the restricted map $\sF|_E\to T^1_{X/k}|_E$ factors through $T^1_{E/k}$, or equivalently if $\sF(I_E)\subset I_E$ at the generic point of $E$.
\end{definition}

It is convenient to introduce the function $\epsilon_{\sF}$ on the set of prime divisors on $X$, defined as follows:
		\begin{equation}\label{eqn:epsilon_fct}
		\epsilon_\sF(E)=\begin{cases}
		0 & \text{if }E\text{ is }\sF\text{-invariant,}\\
		1 & \text{otherwise.}
		\end{cases}
		\end{equation}
If $\pi\colon Y\to X$ is birational with $Y$ normal and $E$ is a prime divisor on $Y$, then we set $\epsilon_\sF(E)=\epsilon_{\pi^*\sF}(E)$. This depends only on the divisorial valuation defined by $E$ on $K(X)$, not on $\pi$. We drop the subscript and write $\epsilon(\bullet)$ is no confusion is likely to arise.

\begin{remark}
Suppose that $Z\subset X$ is a closed subset, and that $X$ is regular at the generic point of $Z$. If $\pi\colon \Bl_ZX\to X$ is the blow-up, then the (non-)invariance of the (unique) $\pi$-exceptional divisor $E$ can be a subtle question, already on surfaces.
	\begin{enumerate}
		\item If $\sF\subsetneq T_{\bA^2/k}$ is regular at the origin, then the blow-up of the origin will produce an invariant divisor;
		\item If $\sF\subsetneq T_{\bA^2/k}$ is not regular at the origin, then the blow-up of the origin may or may not produce an invariant divisor: see \autoref{example:toric_derivation_not_canonical}.
	\end{enumerate}
In particular, $\sF(I_Z)\subset I_Z$ does not guarantee that $E$ is $\pi^*\sF$-invariant (because of the saturation involved in defining $\pi^*\sF$).
\end{remark}

\subsection{Infinitesimal quotients}
For simplicity of exposition, let us discuss quotients by derivations before quotients by foliations.

\subsubsection{Quotients by derivations}
Let $A$ be a $k$-algebra and $D\in\Der_k(A)$. 

\begin{definition}
The \textbf{subring of constants} of $D$ is the subset $A^D=\{a\in A\mid D(a)=0\}$. 
\end{definition}

It is easily seen that $A^D$ is a subring of $A$, and that it contains $k[A^p]$. 

\begin{lemma}\label{lemma:quotient_is_normal}
Let $A$ and $D$ be as above, and assume that $A$ is integral. Then:
	\begin{enumerate}
		\item $\Frac(A)^D=\Frac(A^D).$
		\item If $A$ is $F$-finite or normal, so is $A^D$.
		\item If $x,y\in A^D$ is a regular sequence in $A$, then $x,y$ is also a regular sequence in $A^D$.
	\end{enumerate}
\end{lemma}
\begin{proof}
The derivation $D$ extends to $\Frac(A)$ following the usual rule for differentiating quotients. Clearly $\Frac(A^D)\subseteq\Frac(A)^D$. Conversely, assume that $\frac{a}{b}\in\Frac(A)^D$. As $\frac{a}{b}=\frac{ab^{p-1}}{b^p}$ we have
		$$0=D\left(\frac{a}{b}\right)=\frac{D(ab^{p-1})}{b^p}$$
	and so $D(ab^{p-1})=0$. Therefore $\frac{ab^{p-1}}{b^p}\in\Frac(A^D)$, showing the first point.
	
Assume that $A$ is $F$-finite. Then $A$ is a finite module over $A^p$, a fortiori over $A^D$: by Artin--Tate lemma \cite[00IS]{Stacks_Project} it follows that $A^D$ is finite over $A^p$. Since $A^p$ is finite over $(A^D)^p$, we obtain that $A^D$ is $F$-finite. 

Assume that $A$ is normal, and suppose that $x\in \Frac(A^D)$ satisfies a monic polynomial equation with coefficients in $A^D$. Then $x\in A$ by normality of $A$, and $D(x)=0$ by assumption. So $x\in A^D$, showing that $A^D$ is normal.

Finally, assume that $x,y\in A$ is a regular sequence. 
%By \cite[Theorem 16.1]{Matsumura_Commutative_Ring_Theory} the sequence $x^p,y^p\in A$ is regular. 
Clearly $x$ is not a zero-divisor in the subring $A^D$. Now assume that multiplication by $y$ is not injective on $A^D/xA^D$. Then we have an equality $zy=wx$ with $z,w\in A^D$ and $z\notin xA^D$. Considering this equality in $A$, we must have $z=ax$ for some $a\in A$. But $0=D(z)=xD(a)$ hence $a\in A^D$ and so in fact $z\in xA^D$: contradiction. Hence $x,y$ is a regular sequence on $A^D$.
\end{proof}

The singularities of $A^D$ are difficult to describe beyond this lemma, even if $A$ is regular. We refer to \cite{Aramova_Avramov_Singularities_of_quotients_by_vector_fields, Aramova_Reductive_derivations}
for a  general discussion, and to \cite{Miyanishi_Russell_Purely_insep_cov_of_affine_plane} for several two-dimensional examples. We will only be interested in the cases where $D$ is $p$-closed, but this does not simplify the matter by much. If $D$ is additive, all hell may break loose: the following proposition shows that every singularity universally homeomorphic to a regular point is a composition of $\alpha_p$-quotients.

\begin{proposition}\label{thm:locally_alpha_2_quotient}
Let $f\colon X\to Y$ be a finite purely inseparable morphism of normal $\bF_p$-schemes of degree $p$. Then $f$ is locally an $\alpha_p$-quotient.
\end{proposition}
\begin{proof}
Indeed, $f$ is locally the quotient by a $p$-closed derivation \cite[Proposition 2.4]{Matsumoto_Purely_inseparable_coverings_of_RDP}. Now apply \autoref{lemma:scaling_is_additive}.
\end{proof}

The following lemma illustrates the typical singularities that may arise.

\begin{lemma}\label{rmk:additive_quotients_are_only_S2}
Assume that $A$ is a local normal $k$-algebra of dimension $d\geq 3$ and that $D\in \Der_k(A)$ is an additive derivation. Suppose that $D(A)$ generates an $\fm_A$-primary ideal.
%and that the induced linear map $\bar{D}\in \End_k(A/\fm_A)$ is zero. 
Then $A^D$ is not $S_3$ and not $F$-injective.
\end{lemma}
\begin{proof}
The following argument is essentially contained in \cite[\S 5]{Liedtke_Martin_Matsumoto_Torsors_over_RDP}. Let $Y=\Spec(A)$. By \autoref{prop:derivations_and_group_actions} and \autoref{lemma:fixed_locus_of_action}, the derivation $D$ defines an action of $\alpha_p$ on $Y$ which is free on $Y^*=Y\setminus \{\fm_A\}$ and whose reduced fixed locus is $\{\fm_A\}$. Let $X=\Spec(A^D)=Y/\alpha_p$ be the geometric quotient: it is also a local normal affine scheme of dimension $d$. Let $q\colon Y\to X$ be the quotient map, let $\fn=q(\fm_A)$ and write $X^*=X\setminus\{\fn\}$.

Recall that $\alpha_p$-torsors are classified by the flat cohomology groups $H^1_\text{fl}(\bullet, \alpha_p)$. The restriction $q^*\colon Y^*\to X^*$ is an $\alpha_p$-torsor, and thus defines an element $[q^*]\in H^1_\text{fl}(X^*,\alpha_p)$. We claim $[q^*]$ does not belong to the natural restriction map $r\colon H^1_\text{fl}(X,\alpha_p)\to H^1_\text{fl}(X^*,\alpha_p)$. Indeed, suppose that there was an $\alpha_p$-torsor $\mathfrak{q}\colon Y'\to X$ such that $\mathfrak{q}\times_YY^*=q^*$. Then $Y'$ is affine and $S_2$. Since $\mathfrak{q}$ is finite we see that $\mathfrak{q}^{-1}(\fn)$ has codimension $d$ and so
		$$Y'=\Spec H^0(Y',\sO_{Y'})=
		\Spec H^0(Y'\setminus \mathfrak{q}^{-1}(\fn),\sO_{Y'})
		\cong\Spec H^0(Y^*,\sO_{Y})
		=\Spec H^0(Y,\sO_Y)$$ 
where, for the last equality, we used that $Y$ is $S_2$ as well. Therefore $q=\mathfrak{q}$: but this is impossible since $q$ is not an $\alpha_p$-torsor.

Now we relate the non-surjectivity of $H^1_\text{fl}(X,\alpha_p)\to H^1_\text{fl}(X^*,\alpha_p)$ to local cohomology. Evaluate the exact sequence of flat group schemes
		$$0\to \alpha_p\to \bG_a\overset{F_{\bG_a/k}}{\longrightarrow}\bG_a \to 0$$
on $X^*$ and $X$. Taking in account that $H^i_\text{fl}(\bullet,\bG_a)=H^i(\bullet,\sO_{\bullet})$, the beginning of the long exact sequence of cohomology gives the commutative diagram
		$$\begin{tikzcd}
		0 \arrow[r] & \coker\left( F_{X/k}|_{H^0(X,\sO_{X})}\right)\arrow[d, "="]\arrow[r] & H^1_\text{fl}(X,\alpha_p) \arrow[d, "r"]\arrow[r] & 0 \\
		0 \arrow[r] & \coker\left( F_{X^*/k}|_{H^0(X^*,\sO_{X^*})}\right) \arrow[r] & H^1_\text{fl}(X^*,\alpha_p) \arrow[r] & \ker\left( F_{X^*/k}|_{H^1(X^*,\sO_{X^*})}\right) \arrow[r] & 0 
		\end{tikzcd}$$
with exact rows. Here we used that $H^1(X,\sO_X)=0$ since $X$ is affine, and the left-most vertical arrow is an equality because $X$ is normal and $X^*\subset X$ is big. So the non-surjectivity of $r$ implies that the Frobenius action on $H^1(X^*,\sO_{X^*})$ has a non-trivial kernel. Looking at the usual long exact sequence 
		$$H^i_\fn(X,\sO_X)\to H^i(X,\sO_X)\to H^i(X^*,\sO_{X^*})\overset{+1}{\longrightarrow},$$
on which the Frobenius acts compatibly, we deduce that the action of the Frobenius on $H^2_\fn(X,\sO_X)\cong H^1(X^*,\sO_{X^*})$ has a non-trivial kernel. This means that $X$ is not $S_3$ neither $F$-injective.
\end{proof}

The quotients by multiplicative derivations are, in comparison, way nicer: we discuss these in \autoref{section:quotients_by_mult_der}.  

Let us compute two examples of subring of constants.

\begin{example}\label{example:ring_of_csts_toric_der}
Consider the derivation $\partial_{a,b}$ on $k[x,y]$ from \autoref{example:toric_derivation}. Then we have $k[x,y]^{\partial_{a,b}}=k[x^iy^j\mid ai+bj=0 \ (p)]$.
\end{example}

\begin{example}\label{example:ring_of_csts}
Consider $D=x^p\frac{\partial}{\partial x}+y^p\frac{\partial}{\partial y}$ on $k[x,y]$. Then $D^{[p]}=0$. Clearly $x^p,y^p$ and $x^py-xy^p$ belong to the ring of constants, and I claim that they generate it. It suffices to show that
		$$k[x^p,y^p,x^py-xy^p]\cong k[u,v,s]/(s^p-(u^2v-uv^2))$$
is a normal ring, which is easily seen using the Jacobian criterion. 

If $p=2$, then in Artin's terminology \cite{Artin_Coverings_of_RDPs_in_pos_char} this is a $D^0_{4}$ rational double point (hence a canonical singularity). Using Fedder's criterion, one checks that it is not $F$-pure.
\end{example}

\subsubsection{Quotients by foliations}
Let $X$ be a normal Noetherian connected $k$-scheme and let $\sF\subset \Der_k(\sO_X)$ be a foliation on $X$. We define the presheaf $\sO_X^\sF$ on the topological space $|X|$ by
		$$\sO_X^\sF(U)=\{s\in \sO_X(U)\mid D(s)=0 \ \forall D\in\sF(U)\}.$$
This is a sheaf of algebras on $|X|$. It is easy to see that the locally ringed space $(|X|,\sO_X^\sF)$ is a $k$-scheme: for if $\Spec(A)$ is an affine chart affine, then 
		$$k[A^p]\hookrightarrow A^\sF\hookrightarrow A$$
so $\Spec(A^\sF)$ has the same underlying topological space as $\Spec(A)$.

\begin{definition}
The \textbf{quotient} of $X$ by $\sF$ is the $k$-scheme $X/\sF=(|X|,\sO_X^\sF)$.
\end{definition}

In particular $X/\sF$ comes with a purely inseparable morphism $q\colon X\to X/\sF$ which is a universal homeomorphism and factors the $k$-linear Frobenius of $X$ (\footnote{
		The varieties that factor an iterated Frobenius morphism of regular varieties are sometimes called \emph{Frobenius sandwiches} in the literature.
}). If $X$ is $F$-finite then $q$ is a finite morphism. In this article, we will only consider the case where $\sF$ is a $1$-foliation.

The following well-known lemma will be used implicitly many times. 
\begin{lemma}[{cf \cite[\S 4.1]{Tziolas_Quotient_by_alpha_p_and_mu_p}}]\label{lemma:localisation_properties}
Let $A$ be a normal Noetherian $F$-finite $k$-algebra and $\sF$ a foliation on $\Spec(A)$. If $\mathfrak{p}$ is a prime ideal of $A$ with contraction $\mathfrak{q}\subset A^\sF$, then:
	\begin{enumerate}
		\item $A^\sF$ is Noetherian,
		\item $(A^\sF)_\mathfrak{q}=(A_\mathfrak{p})^{\sF_\mathfrak{p}}$, and
		\item if $\widehat{A}=\widehat{A_\mathfrak{p}}$ and $\widehat{\sF}=\sF_\mathfrak{p}\otimes \widehat{A}$, then $\widehat{A}^{\widehat{\sF}}$ is a local ring and it is equal to the completion of $(A^\sF)_\mathfrak{q}$.
	\end{enumerate}
\end{lemma}
\begin{proof}
First we show that $A^p$ is Noetherian: we have to show that any ideal $I\subset A^p$ is finitely generated. Let $\sI=\{f\in A\mid f^p\in I\}$. Then clearly $\sI$ is an ideal of $A$, and so it is finitely generated, say $\sI=(f_1,\dots,f_r)$. Then it is easily seen that $I=(f_1^p,\dots,f_r^p)$. So $A^p$ is Noetherian.

As $A^p$ is Noetherian and $A$ is $F$-finite, we get that $A$ is a Noetherian $A^p$-module. As $A^\sF$ is a sub-$A^p$-module, it must be a finite $A^p$-module. The Noetherian ring property of $A^p$ now ascends to $A^\sF$ by the Artin--Tate lemma \cite{Artin_Tate_Note_on_finite_extensions}.

The localization property is \cite[Proposition 4.2]{Tziolas_Quotient_by_alpha_p_and_mu_p}. For the completion property, we may assume that $A$ is local with maximal ideal $\mathfrak{p}$, so $B=A^\sF$ is also local with maximal ideal $\mathfrak{q}$. Since $A$ is $F$-finite, it is a finite $B$-module. As $\sF$ is a finite $A$-module, it is also a finite $B$-module. So $\widehat{A}=A\otimes_B\widehat{B}$ and $\widehat{\sF}=\sF\otimes_B\widehat{B}$. If $D_1,\dots,D_n$ generate $\sF$, then we have an exact sequence
		$$
		0\to  B \to  A \overset{\bold{D}}{\longrightarrow}  A^{\oplus n},\quad \text{where  } \bold{D}(a)=(D_1(a),\dots,D_n(a))
		$$
which remains exact after tensoring by $\widehat{B}$, by flatness of completion. As $\sum_i D_i\otimes_B\widehat{B}=\widehat{\sF}$, we obtain $\widehat{B}=\widehat{A}^{\widehat{\sF}}$.
\end{proof}

\begin{lemma}\label{lemma:quotient_is_normal_II}
Notations as above. 
	\begin{enumerate}
		\item $K(X)^\sF=K(X/\sF).$
		\item If $X$ is $F$-finite and normal, so is $X/\sF$.
		\item If $s,t\in \sO_{X/\sF,q(x)}$ form a regular sequence in $\sO_{X,x}$, then $s,t$ is also a regular sequence in $\sO_{X/\sF,q(x)}$.
	\end{enumerate}
\end{lemma}
\begin{proof}
The question is Zariski-local, so we may assume that $X=\Spec(A)$. Say that $\sF$ is generated by $D_1,\dots,D_n\in\Der_k(A)$. Then $X/\sF=\Spec\left(\bigcap_iA^{D_i}\right)$. So the assertions follow from \autoref{lemma:quotient_is_normal} applied inductively on $A_0=A, A_1=A_0^{D_1}, A_2=A_1^{D_2}$, etc., once we have observed that in the proof of \autoref{lemma:quotient_is_normal} there is no harm in assuming that the codomain of $D$ was the fraction field.
%, and that raising the two elements to the $p$-power in the last assertion is unnecessary if they already belong to the subring of constants. 
\end{proof}

\begin{lemma}[{\cite[Part I, III.1.9]{Miyaoka_Peternell_Geometry_of_higher_dim_varieties}}]\label{lemma:reg_quotient_implies_reg_foliation}
Let $X$ be a smooth $k$-scheme, $\sF$ a $1$-foliation on $X$. Then $X/\sF$ is regular if and only if $\sF$ is regular.
\end{lemma}
\begin{proof}
We may base-change to the algebraic closure of $k$ and localize at a closed point, so that $X=\Spec(A)$ is the spectrum of a regular local $k$-algebra with algebraically closed residue field $k$. 
If $\sF$ is regular, then by \autoref{example:reg_foliations_on_reg_varieties} we can find coordinates $x_1,\dots,x_n$ of $A$ such that 
		$$\sF=\sum_{i=1}^rA\cdot \frac{\partial}{\partial x_i} \quad \text{for some }r\leq n.$$
Then clearly $\widehat{A}^{\widehat{\sF}}\cong k\llbracket x_1^p,\dots,x_r^p,x_{r+1},\dots,x_n\rrbracket$, so $A^\sF$ is regular.

Conversely, assume that $A^\sF$ is regular. We have to check that $T_{A/k}/\sF$ is free. In fact, by Kunz's conjecture---a theorem since 1981, see \cite{Kimura_Niitsuma_On_Kunz_conjecture}---, $A^\sF$ has a $p$-basis $x_{r+1},\dots,x_n$ over $A^p$. Then $x_{r+1}^p,\dots,x_n^p$ extend to a set of local coordinates $x_1^p,\dots,x_n^p$ of $A^p$. From this and the Jacobson correspondence it follows immediately that $\sF=\bigoplus_{i=1}^rA\frac{\partial}{\partial x_i}$ and thus the quotient $T_X/\sF=\bigoplus_{i=r+1}^nA\frac{\partial}{\partial x_i}$ is free. 
\end{proof}

\begin{remark}
The above lemma, together with its proof, should be compared to the following result of Zariski \cite[Lemma 4]{Zariski_Studies_in_equisingularity_I} (see also the more general Nagata--Zariski--Lipman theorem \cite[Theorem 30.1]{Matsumura_Commutative_Ring_Theory}): if $(A,\fm)$ is a complete local ring of characteristic $0$, and $D\in\Der(A)$ such that $D(x)\notin \fm$ for some $x\in A$, then there exists a subring $A'\subset A$ over which $x$ is analytically independent, such that $A=A_1\llbracket x\rrbracket$ and $D|_{A_1}\equiv 0$. Notice also that in this case the assumption that $A$ is complete is necessary, whereas in the proof of \autoref{lemma:reg_quotient_implies_reg_foliation} completion is not needed to find a system of parameters adapted to $\sF$.
\end{remark}

For the formation of quotients, the facts that $\sF$ is saturated, closed under Lie brackets and $p$-th powers are irrelevant. However, if we restrict to quotients by $1$-foliations, we obtain the following geometric meaningful correspondence.

\begin{theorem}[{Jacobson correspondence: \cite{Jacobson_Galois_theory_pur_insep_expon_one}, \cite{Patakfalvi_Waldron_Sing_of_general_fibers_and_LMMP}}]\label{thm:Jacobson_correspondence}
Let $X$ be a normal $k$-variety. Then there is a bijection between
	\begin{enumerate}
		\item $1$-foliations of rank $r$ on $X$, and
		\item factorizations of the $k$-linear Frobenius
				$F_{X/k}\colon X\overset{f}{\longrightarrow} Y\to X^{(-1)}$
			where $Y$ is normal and $\deg(f)=p^r$.
	\end{enumerate}
The bijection is explicitly given by $\sF\mapsto (X\to X/\sF)$ and $(X\to Y)\mapsto \Der_{\sO_Y}(\sO_X)$.
\end{theorem}

Let us also mention the well-known adjunction formula for quotients.

\begin{definition}
Let $\sF$ be a foliation on $X$. We let $K_\sF$ denote the divisor class of the divisorial sheaf $\det(\sF)^{[-1]}$. By abuse of terminology we call $K_\sF$ the \textbf{canonical divisor} of $\sF$.
\end{definition}

\begin{proposition}[{Adjunction formula: \cite[Proposition 2.10]{Patakfalvi_Waldron_Sing_of_general_fibers_and_LMMP}}]\label{prop:adjunction_formula}
Let $q\colon X\to X/\sF$ be the quotient of a normal $k$-variety by a $1$-foliation. Then we have a sequence of $\sO_X$-modules
		$$0\to \sF\to T_{X/k}\to q^{[*]}T_{X/\sF}\to F^{[*]}_{X/k}\sF\to 0$$
which is exact on a big open subset of $X$. Consequently we have an equality
		$$q^*K_{X/\sF}=K_X+(p-1)K_\sF$$
of divisor classes inside $\mathrm{Cl}(X)$.
\end{proposition}

\section{Birational singularities of $1$-foliations}\label{section:MMP_sing}
Given a foliation $\sF$ on a normal $k$-variety $X$, we define the birational singularities of $\sF$ in the spirit of the MMP (see \cite{Kollar_Mori_Birational_geometry_of_algebraic_varieties}).  We formulate the definition for any foliation, but we will only use it for $1$-foliations. It is convenient to do so in presence of a $\bQ$-Weil divisor $\Delta$. 

\begin{definition}
We say that $(\sF,\Delta)$ is a \textbf{($\bQ$-)Gorenstein foliated pair} if $K_\sF+\Delta$ is ($\bQ$-)Cartier.
\end{definition}

Assume that $(\sF,\Delta)$ is $\bQ$-Gorenstein. Then if $\pi\colon Y\to X$ is a birational proper $k$-morphism of normal varieties, we can write
		$$K_{\pi^*\sF}+\pi_*^{-1}\Delta=\pi^*(K_\sF+\Delta)+\sum_E a(E;\sF,\Delta) E$$
where $E$ runs through the $\pi$-exceptional prime divisors
(\footnote{
	As in the usual case, even if $K_\sF$ and $K_{\pi^*\sF}$ are only divisor classes, the numbers $a(E;\sF,\Delta)$ are well-defined. The point is that we have a canonical comparison map $\pi^{-1}\sF\otimes_{\pi^{-1}\sO_X}\sO_Y\to \pi^*\sF$ (where the left-hand side is the module pullback, and the right-hand side is the birational pullback as $1$-foliation).
}). Recalling the function $\epsilon$ from \autoref{eqn:epsilon_fct}, we can make the following definition:

\begin{definition}[{\cite{McQuillan_Canonical_models_foliations, Spicer_High_dim_foliated_Mori_theory}}]
Suppose that $(\sF,\Delta)$ is $\bQ$-Gorenstein. Then $(\sF,\Delta)$ is:
	\begin{enumerate}
		\item \textbf{terminal} (resp. \textbf{canonical}) if $a(E;\sF,\Delta)>0$ (resp. $a(E;\sF,\Delta)\geq 0$) for all exceptional prime divisors $E$ over $X$;
		\item \textbf{klt} if $\lfloor \Delta\rfloor=0$ and  $a(E;\sF,\Delta)> -\epsilon_{\sF}(E)$ for all exceptional $E$ over $X$;
		\item \textbf{log canonical (lc)} if $a(E;\sF,\Delta)\geq -\epsilon_{\sF}(E)$ for all exceptional $E$ over $X$.
	\end{enumerate}
\end{definition}

Some remarks are in order. It follows from the definitions that the following implications hold:
	$$
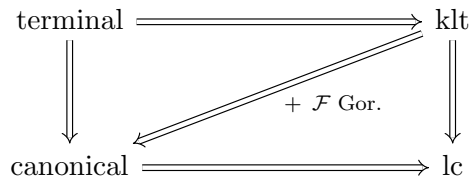
\begin{tikzcd}
	\text{terminal} \arrow[rrrr, Rightarrow]\arrow[dd, Rightarrow] &&&& \text{klt} \arrow[dd, Rightarrow]\arrow[ddllll, Rightarrow, "(\sF{,}\Delta)\text{ Gor.}"] \\
	&&&&\\
	\text{canonical} \arrow[rrrr, Rightarrow] &&&& \text{lc}
	\end{tikzcd}$$
However, if $\sF$ is canonical then \emph{a priori} it is not necessarily klt.

The terminal and klt conditions seem rather restrictive for foliations, especially so in positive characteristic. It will follow from \autoref{prop:can_foliations_on_surfaces} that if a $1$-foliation $\sF$ is terminal at a point $x\in X$ then $\sO_{X,x}$ must be singular. In particular the terminal locus of $1$-foliations is never open.

Moreover, in positive characteristic the birational singularities of $X$ and $\sF$ may have little in common\footnote{This contrasts with the characteristic $0$ case: see 
\cite[Remark 3.10]{Spicer_High_dim_foliated_Mori_theory} and \cite[Theorem 1.5]{Cascini_Spicer_MMP_corank_1_foliations_on_3folds}.}: see \cite[Example 2.14]{Bernasconi_Counterexample_MMP_for_foliations_in_pos_char}.

%To avoid confusions with the similar notions defined for a pair $(X,\Delta)$, we sometimes say that $(\sF,\Delta)$ is \textbf{foliated terminal}, etc.

As an example, we describe the birational singularities of regular $1$-foliations on regular varieties.

\begin{lemma}\label{lemma:regular_foliations_are_can}
Regular $1$-foliations on regular varieties are canonical.
\end{lemma}
\begin{proof}
We use an observation of Bernasconi \cite[2.12-13]{Bernasconi_Counterexample_MMP_for_foliations_in_pos_char}.  The statement is \'{e}tale-local, so it suffices to consider the case $X=\bA^n_\bold{x}$ and $\sF=\bigoplus_{i=1}^r\sO_X\cdot \frac{\partial}{\partial x_i}$. Then $K_\sF$ is the line bundle generated inside $\bigwedge^rT_X^1$ by the $r$-vector field $\theta$ dual to the $r$-form $dx_1\wedge\dots\wedge dx_r\in \Omega_X^r$. If $f\colon Y\to X$ is birational and $\tilde{\theta}$ is the pullback of $\theta$, then $f^*(dx_1\wedge\dots\wedge dx_r)(\tilde{\theta})=1$. As $f^*(dx_1\wedge\dots\wedge dx_r)$ is a regular $r$-form, we see that $\tilde{\theta}$ does not have any zero along exceptional divisors. This implies that $K_{f^*\sF}-K_{\sF}\geq 0$, and so $\sF$ is canonical.
\end{proof}

\begin{lemma}\label{example:reg_foliations_not_terminal}
Non-trivial regular $1$-foliations on regular varieties are not terminal.
\end{lemma}
\begin{proof}
Let $X$ be a regular variety and $\sF$ be a non-trivial regular $1$-foliation. The result we want to establish is local, so we may work at the stalk $\sO_{X,s}$ of the variety at a closed point. By \autoref{example:reg_foliations_on_reg_varieties}, we may find local coordinates $x_1,\dots,x_n\in \sO_{X,s}$ such that $\sF$ is generated by $\partial/\partial x_i$ for $i=1,\dots,r$ (where $1\leq r<n$). 

If $r\leq n-2$ (that is, $\sF$ has corank $>$1), we can blow-up the ideal $(x_{n-1},x_n)$ to obtain an exceptional divisor along which $\sF$ has zero discrepancy. 

A construction that works also when $\sF$ has corank $1$ is the following (cf.\ \cite[Example 2.13]{Bernasconi_Counterexample_MMP_for_foliations_in_pos_char}). Take the weighted blow-up of the Rees algebra $(x_1,1)+(x_n,p)$: a schematic chart $\bA^n_{\bar{x}_\bullet}$ is given by
		$$x_1=\bar{x}_1, \quad x_n=\bar{x}_1^p\bar{x}_n,
		\quad x_i=\bar{x}_i \ \text{ for }i=2,\dots,n-1.$$
Then $\partial/\partial x_i$ pullbacks to $\partial/\partial \bar{x}_i$ for $i=1,\dots,r$. As $\sO(-K_\sF)$ is generated by $\bigwedge_{i=1}^r\partial/\partial x_i$, and $\sO(-K_{\overline{\sF}})$ is generated by $\bigwedge_{i=1}^r\partial/\partial \bar{x}_i$, where $\overline{\sF}$ is the birational pullback of $\sF$, we obtain that the discrepancy of $\sF$ along the exceptional divisor $(\bar{x}_1=0)$ is zero.
\end{proof}

\begin{example}\label{example:toric_derivation_not_canonical}
Consider the $1$-foliation $\sF_{a,b}$ on $\bA^2$ introduced in \autoref{example:sing_of_toric_derivation}, and assume that $ab\neq 0$. \emph{I claim that it is log canonical, but not canonical nor klt.} Let us show that these are not canonical nor klt; the log canonical property can be obtained by similar considerations, using that any divisorial valuation on $\bA^2$ can be reached by a sequence of blow-ups of points \cite[2.45]{Kollar_Mori_Birational_geometry_of_algebraic_varieties}, or alternatively by using \autoref{prop:lc_foliation_and_linear_alg}.

We may assume that $a\neq 0$. Let us blow-up the origin and consider the chart given by
		$$\pi\colon \bA^2_{u,v}\to\bA^2_{x,y}, \quad (u,v)\mapsto (u,uv).$$
Explicit computations show that $\partial_{a,b}$ lifts to $\partial_{a,b-a}$ on $\bA^2_{u,v}$.
	\begin{enumerate}
		\item Assume that $a=b$, so we may assume that both are equal to $1$. Then the induced (saturated) $1$-foliation $\pi^*\sF_{1,1}$ is $\sF_{1,0}$, generated by $\frac{\partial}{\partial u}$, and we have $\pi^*K_{\sF_{1,1}}=K_{\sF_{1,0}}+E$ where $E=(u=0)$ is the exceptional divisor. Since $E$ is not invariant, this shows that $\sF_{1,1}$ is not canonical and also not klt, and that $E$ is a log canonical place for the foliated pair $(\bA^2,\sF_{1,1})$.
		\item Next assume that $a\neq b$. Then no saturation is needed to obtain the induced $1$-foliation $\pi^*\sF_{a,b}=\sF_{a,b-a}$. In particular $\pi^*K_{\sF_{a,b}}=K_{\sF_{a,b-a}}$. The exceptional divisor is invariant for $\sF_{a,b-a}$, which shows that $\sF_{a,b}$ is not klt. 
		
		Let $n>0$ be minimal such that $b-na=0$ in $\bF_p$. If we iterate this blow-up procedure $n$ times, we find a birational morphism $\varphi\colon\bA^2\to \bA^2$ such that $\varphi^*K_{\sF_{a,b}}=K_{\sF_{1,0}}+E$ where $E$ is a prime $\varphi$-exceptional divisor which is invariant for $\sF_{1,0}$. So $\sF_{a,b}$ is not canonical. 
	\end{enumerate}
\end{example}

%\begin{remark}
%Cascini-Spicer, corank 1 MMP, rmk 2.3
%\end{remark}

In case $X$ is regular and $\sF$ is a line bundle, there is a simple local characterization of log canonical foliations due to McQuillan. 

\begin{proposition}\label{prop:lc_foliation_and_linear_alg}
Suppose $(x\in X)$ is the spectrum of a regular local ring with maximal ideal $\fm$ and residue field $k(x)$, and $\sF$ is an invertible $1$-foliation generated by $\partial$. Then:
	\begin{enumerate}
		\item if $\partial\notin \fm\Der_k(\sO_{X,x})$ then $\sF$ is canonical at $x$;
		\item if $\partial\in \fm\Der_k(\sO_{X,x})$, then $\sF$ is log canonical at $x$ if and only if $k(x)$-linear endomorphism $\partial_0\in\End_{k(x)}(\fm/\fm^2)$ is not nilpotent.
	\end{enumerate}
\end{proposition}
This is stated in \cite{McQuillan_Canonical_models_foliations}, but as locating the precise argument is a non-trivial task we sketch the proof for the convenience of the reader.
\begin{proof}
The first part follows immediately from \autoref{lemma:singularity_Gor_foliation} and \autoref{lemma:regular_foliations_are_can}. Thus we may assume that $\partial\in \fm\Der_k(\sO_{X,x})$.

First, observe that $\sO_X\cdot \partial=K_\sF^{-1}$. So if $E$ is any exceptional prime divisor appearing on $\pi\colon Y\to X$, then at the generic point of $E$ we can write $\pi^*\partial= t_E^{-a(E;\sF)}\partial'$, where $t_E$ is a uniformizer of $\sO_{Y,E}$ and $\partial'\in T_{Y/k}- \fm_E T_{Y/k}$. The number $a(E,\sF)$ is necessarily an integer. 

\textsc{Non-log canonicity implies nilpotence.} Let $E$ be as above, centered at $x$, and assume that $a(E,\sF)<-\epsilon_\sF(E)$. Write $e=-a(E;\sF)$. We distinguish two cases:
	\begin{enumerate}
		\item Assume that $e\geq 2$. Then $\pi^*\partial(\fm_E^n)\subset \fm_E^{n+e-1}$. Letting $\fp_n=\fm_E^n\cap \sO_{X,x}$, we obtain $\partial(\fp_n)\subset \fp_{n+e-1}$. As
		$$\bigcap_{n}\fp_n\subseteq \bigcap_n\fm_E^n=\{0\},$$
	it follows from \cite[Lemma 7 p.695]{Chevalley_Local_rings} that the $\fp_n$'s generate a finer topology than the $\fm_{X,x}$-adic topology. On the other hand $\fm_{X,x}^n\subseteq \fm_E^n$ and so $\fp_n\supseteq \fm_{X,x}^n$. Hence the $\fp_n$'s generate the $\fm_{X,x}$-adic topology. We deduce that the continuous extension of $\partial$ to $\widehat{\sO}_{X,x}$ is topologically nilpotent. In particular, $\partial^{[N]}(\fm_{X,x})\subset \fm_{X,x}^2$ for some $N\gg 0$. This implies that $\partial_0^{[N]}=0$.
		\item The remaining case is $e=1$. Then $E$ must be invariant for $\pi^*\sF$, which means that $\partial'(\fm_E)\subset \fm_E$. Then $\partial(\fp_n)\subset \fp_{n+1}$ and the same argument gives that $\partial_0$ is nilpotent.
	\end{enumerate}

\textsc{Nilpotence implies non-log canonicity.} We also distinguish two cases, in which we explicitly produce a divisor with discrepancy $\leq -2$. The main tool is the formulas of \autoref{example:smth_blow_up}.
	\begin{enumerate}
		\item $\partial_0=0$. Then $\partial \in \fm^2T_X$, where $\fm$ is the maximal ideal of the closed point $x$. The blow-up of the closed point produces an exceptional divisor $E$ whose discrepancy is $< -\epsilon(E)$. To see this, we may work \'{e}tale-locally and reduce to the case where we blow-up $\bA^n_{x_1,\dots,x_n}$ along the ideal $I=(x_1,\dots,x_l)$ (with $l\leq n$), and $\partial=\sum_{i=1}^nf_i\partial_{x_i}$ satisfies $\min_i \ord_If_i=d\geq 2$. If $\pi\colon \bA^n_{\bold{y}}\to \bA^n_{\bold{x}}$ is the $y_1$-patch of the blow-up, then we have
				$$\pi^*\partial = 
				y_1^{d-1}\left[\tilde{f}_1(\bold{y})\partial_{y_1}
				+ \sum_{i=2}^l \frac{1}{y_1}\left(\tilde{f}_i(\bold{y})- y_i\tilde{f}_1(\bold{y})\right) \partial_{y_i}
				+\sum_{j=l+1}^n \tilde{f}_j(\bold{y})\partial_{y_j}\right]
				$$
		where the $\tilde{f}_\bullet(\bold{y})=y_1^{-d+1}(f_\bullet(\bold{x})\circ\pi)$ are regular functions that are still divisible by $y_1$. If the exceptional divisor $E=(y_1=0)$ is invariant, then $a(E;\sF)\leq -d+1<0=-\epsilon(E)$. Assume that $E$ is not invariant: as $\tilde{f}_1(\bold{y})|_{(y_1=0)}=0$, this implies that the derivation in brackets is still divisible by $y_1$, and therefore $a(E;\sF)\leq -d<-\epsilon(E)$.
		\item If $\partial_0\neq 0$ is nilpotent, we will show that after a few well-chosen blow-ups we reduce to the previous case. The Jordan normal form of $\partial_0$ is defined over $k(x)$, since its only eigenvalue is $0$. So we may choose local coordinates with respect to which the matrix of $\partial_0$ is already in Jordan normal form. The centers of the blow-ups will depend only on these coordinates. Thus we may again work \'{e}tale-locally on $\bA^n_{x_1,\dots,x_n}$, say that the point $x$ corresponds to the ideal $I=(x_1,\dots,x_l)$ (with $l\leq n$), and assume that a nilpotent non-zero Jordan block is spanned by $x_1,\dots,x_r$ (with $r\leq l$). We are going to blow-up $(x_1,\dots,x_r)$: since this does not affect the other Jordan blocks, we may as well assume that $r=l$. So write
				$$\partial=\sum_{i=1}^{r-1}x_i\partial_{x_{i+1}} +\delta, \quad \delta\in I^2T_{\bA^2/k}.$$
	If $\pi\colon \bA^n_{\bold{y}}\to \bA^n_{\bold{x}}$ is the $y_r$-patch of the $I$-blow-up, then one finds that
			$$\pi^*\partial=
			\sum_{i=1}^{r-2}y_i\partial_{y_{i+1}}
			+
			\left[ y_{r-1}y_r\partial_{y_r}
			-\sum_{i=1}^{r-1}y_{r-1}y_i\partial_{y_i}+\pi^*\delta
			\right].$$
	The derivation in parenthesis belongs to $(y_1,\dots,y_r)^2T_{\bA^n/k}$, and the dimension of the $0$-eigenspace at the origin has increased. Notice also that $\pi^*\partial$ is not divisible by $y_r$, so the discrepancy of the (invariant) $\pi$-exceptional divisor is $0$. By induction, we therefore reduce to the case where $\partial_0=0$. 
	\end{enumerate}
The proof is complete.
\end{proof} 

%McQuillan does at second blow-up, see Non-commutative Mori theory I.1.10, but it seems to me that it is unnecessary.

\begin{corollary}\label{cor:lc_implies_mult_sing}
Let $X$ be a regular variety over $k$, and $\sF$ be a $1$-foliation of rank $1$ on $X$. Then:
	\begin{enumerate}
		\item $\sF$ is canonical if and only if $\sF$ is regular.
		\item $\sF$ is lc if and only if $\sF$ has at worst multiplicative singularities.
	\end{enumerate}
\end{corollary}
\begin{proof}
By \autoref{prop:lc_foliation_and_linear_alg} and \autoref{def:at_worst_mult_sing} this is a formal-local question, so we may assume that $X=\Spec(\sO)$ is the spectrum of a complete regular local ring, with maximal ideal $\fm$, and $\sF$ is generated by $\partial\in\Der_k^\text{cont}(\sO)$ (for $\sF$ is a divisorial sheaf on a regular variety, hence an invertible sheaf). We may assume that $\partial(\fm)\subset \fm$, so we get an induced endomorphism $\partial_0\in\End_{\sO/\fm}(\fm/\fm^2)$. 

First assume that $\sF$ is lc. By \autoref{prop:lc_foliation_and_linear_alg}, $\partial_0$ is not nilpotent. But we have $\partial^{[p]}=u\partial$ for some $u\in\sO$, since $\sF$ is closed under $p$-th powers. In particular $\partial_0^{\circ p}=\bar{u}\partial_0$, where $\bar{u}$ is the image of $u$ in the residue field. Thus $u\notin\fm$. This implies that $\partial$ is multiplicative.

Conversely, if $\sF$ has at worst singularities, we may assume that $\partial$ is multiplicative. Then the function $n\mapsto \partial_0^{\circ n}$ is $p$-periodic. Moreover $\partial_0\neq 0$ by \autoref{rmk:local_descript_mult_sing}, thus $\partial_0$ is not nilpotent and so $\sF$ is lc.

If $\sF$ is regular, we already know it is canonical (\autoref{lemma:regular_foliations_are_can}). Conversely, if $\sF$ is canonical then it is lc and thus has at worst multiplicative singularities. So we may assume that $\sF$ is generated up to saturation by $\sum_i \lambda_i x_i\partial_{x_i}$, where $x_1,\dots,x_n\in\sO$ are formal coordinates and $\lambda_i\in\bF_p$. If more than one $\lambda_i$ is non-zero, then a series of blow-ups as in \autoref{example:toric_derivation_not_canonical} show that $\sF$ is not canonical. Thus a single $\lambda_i$ is non-zero, and it follows that $\sF$ is regular.
\end{proof}

\subsubsection{Amplification of some previous results}
We generalize \autoref{prop:lc_foliation_and_linear_alg} in case the underlying variety $X$ is not regular (\footnote{
	This result does not appear in the printed version of this paper.
}).

\begin{proposition}
Let $(A,\fm)$ be a normal local ring essentially of finite type over $k$, and let $\partial\in \Der_k(A)$ be a non-zero $p$-closed derivation. Let $X=\Spec(A)$ and assume that $\sF=A\cdot \partial$ is a $1$-foliation. Then:
	\begin{enumerate}
		\item If $\partial(A)\not\subseteq \fm$, then $\sF$ is canonical.
		\item If $\partial$ is multiplicative, then $\sF$ is log canonical.
	\end{enumerate}
\end{proposition}
\begin{proof}
First, assume that $\partial(A)\not\subseteq \fm$: then we can find $x\in A$ such that $\partial(x)=u\in A^\times$. Let $f\colon Y\to X$ be a birational morphism from a normal variety, and $E\subset Y$ be a prime divisor. If $\pi\in \sO_{Y,E}$ is a uniformizer, we can write $\partial=\pi^{-a(E;\sF)}\psi$ where $\psi\in \Der_k(\sO_{Y,E})$. Then:
		$$\pi^{-a(E;\sF)}\psi(f^*(x))=f^*(u)\in \sO_{Y,E}^\times,$$
so $a(E;\sF)\geq 0$. This shows that $\sF$ is canonical.

Second, assume that $\partial$ is multiplicative. We may assume that $\partial(\fm)\subseteq \fm$, as otherwise $\sF$ is canonical. Let us fix $i\geq 2$ such that $\partial(\fm)\not\subseteq \fm^i$, and denote by $\widehat{\partial}$ the unique a continuous derivation of $\widehat{A}$ extending $\partial$. If $\sF$ is not log canonical, then the argument in the proof of \autoref{prop:lc_foliation_and_linear_alg} (which works even if $A$ is normal) shows that $\widehat{\partial}$ is topologically nilpotent, i.e.\ $\widehat{\partial}^{\circ N}(\widehat{\fm})\subseteq \widehat{\fm}^i$ for $N\gg 0$. This implies that $\partial$ descends to a nilpotent endomorphism of $\fm/\fm^i$, which is a contradiction with the multiplicativity of $\partial$. So $\sF$ must be log canonical.
\end{proof}

\subsection{Foliations on group quotient singularities}
Let us give a series of examples of lc $1$-foliations on singular varieties. First, consider $f\colon X\to Y$ a finite surjective Galois morphism of normal connected $k$-schemes of finite type. If $G$ is the Galois group of the function field extension $K(Y)\subset K(X)$, then $G$ acts on the $K(X)$-vector space $\Der_k(K(X))$ by
		$$g\cdot \partial=g^{-1}\circ \partial\circ g\in \Der_k(K(X)),\quad
		g\in G, \ \partial\in \Der_k(K(X))$$
and the $K(Y)$-vector field of invariants is $\Der_k(K(Y))$. By Galois descent, the $K(Y)$-subspaces of $\Der_k(K(Y))$ correspond to the $K(X)$-subspaces of $\Der_k(K(X))$ which are stable under the action of $G$. 

\begin{lemma}\label{lemma:pullback_of_foliations_along_quotient}
In the above situation, pullback along $f$ at the generic point gives a rank-preserving bijective correspondence between
	\begin{enumerate}
		\item $1$-foliations on $X$ whose generic stalk is preserved by the $G$-action, and
		\item $1$-foliations on $Y$.
	\end{enumerate}
\end{lemma}
\begin{proof}
This is a direct application of \autoref{rmk:Lie_and_pth_cond_at_generic_point} and Galois descent, except for the fact that if $\sF\subset T_{X/k}$ is a $1$-foliation on $Y$ then $\sF_{K(Y)}\otimes K(X)$ is closed under Lie brackets and $p$-th powers. This is an immediate calculation (using Hochschild's formula \autoref{eqn:Hochschild_formula} for $p$-th powers).
\end{proof}

\begin{definition}
In the above situation:
	\begin{enumerate}
		\item if $\sF\subset T_{Y/k}$ is a $1$-foliation, we let $f^*\sF\subset T_{X/k}$ be the unique $1$-foliation which generic stalk $\sF_{K(Y)}\otimes K(X)$;
		\item if the generic stalk of $\sH\subset T_{X/k}$ is $G$-stable, we let $\sH/G\subset T_{Y/k}$ be the unique $1$-foliation whose generic stalk is $(\sH_{K(X)})^G$.
	\end{enumerate}
\end{definition}

We can now state and prove the following:

\begin{proposition}\label{prop:descent_along_grp_quotient}
Let $\sH$ be a log canonical $1$-foliation of corank $1$ on a normal connected $\bQ$-factorial $k$-scheme $X$ of finite type. Let $G$ be a finite group of order invertible in $k$, acting on $(X,\sG)$. If $(Y=X/G,\sF=\sH/G)$ is the quotient, then $\sF$ is log canonical.
\end{proposition}
\begin{proof}
The quotient $Y$ is $\bQ$-factorial, so $\sF$ is automatically $\bQ$-Gorenstein. So consider a birational proper morphism $f\colon Y'\to Y$, and let $X'$ be the integral closure of $Y'$ in $K(X)$. Then $G$ acts on $X'$ with quotient $Y'$, and we have a commutative diagram
	$$\begin{tikzcd}
	Y & X\arrow[l, "q"] \\
	Y'\arrow[u, "f"] & X'\arrow[l, "q'"]\arrow[u, "f'"]
	\end{tikzcd}$$
where $q,q'$ are the quotient morphisms. If $\sF'$ and $\sH'$ denote the birational pullback of $\sF$ and $\sH$ then by \autoref{lemma:pullback_of_foliations_along_quotient} we have $q^*\sF=\sH$ and $(q')^*\sF'=\sH'$. By the foliated version of Riemann--Hurwitz \cite[Proposition 3.7]{Spicer_High_dim_foliated_Mori_theory} (\footnote{The fact that we are working in positive characteristic does not create any problem: as the order of $G$ is prime to the characteristic, all the ramifications are tame.}) it holds that
		$$K_{\sH}=q^*K_\sF+\sum_{D\subset X}\epsilon_\sH(D)(r_D-1)D, \quad
		K_{\sH'}=(q')^*K_{\sF'}+\sum_{D'\subset X'}\epsilon_{\sH}(D')(r_{D'}-1)D'$$
where $r_{\bullet}$ denotes the ramification index of divisors with respect to the action of $G$.

Let $E$ be a prime $f$-exceptional divisor. We want to estimate $a(E;\sF)$. If $q'$ is \'{e}tale over the generic point of $E$ then $a(E;\sF)=a(E';\sH)\geq 0$ where $E'\subset X'$ is any divisor lying over $E$. So we may assume that $r_{E'}>1$. Notice that $(q')^*E=\sum_{E'}r_{E'}E'$. 

We compute $(q\circ f')^*K_\sF=(f\circ q')^*K_\sF$ at the generic points of $(q')^{-1}(E)$. On the one hand
		$$(q')^*f^*K_{\sF}=K_{\sH'}
		-\sum_{D'\subset X'}\epsilon_{\sH}(D')(r_{D'}-1)D'
		-\sum_{E'}a(E;\sF)r_{E'}E'$$
and on the other
		$$(f')^*q^*K_{\sF}=K_{\sH'}
		-\sum_{E'} a(E';\sH)E'
		-\sum_{D\subset Y}\epsilon_{\sG}(D)(r_D-1)(f')^*D.$$
Equating the two, we find that
		$$r_{E'}a(E;\sF)+\left(r_{E'}-1\right)\epsilon_{\sH}(E')
		= a(E';\sH)+\delta$$
where $\delta\geq 0$. As $\epsilon_{\sH}(E')=\epsilon_\sF(E)$ and $a(E';\sH)\geq -\epsilon_{\sH}(E')$, we find that 
	$$a(E;\sF)\geq -\epsilon_\sF(E)+\frac{\delta}{r_{E'}}$$
which shows that $\sF$ is lc.
\end{proof}

\begin{example}\label{example:quotient_only_lc}
Let $k$ be an algebraically closed field of characteristic $p>2$, and let $G=\bZ/2\bZ$ act $k$-linearly on $\bA^2_{x,y}$ by
		$(x,y)\mapsto (-x,-y)$.
Let $Y=\Spec k[x^2,y^2,xy]$ be the quotient: it is an $A_1$-singularity. If $\sH$ is the $1$-foliation on $\bA^2$ generated by $\partial_x$, then $\sH/G$ is generated by $x\partial_x$ and $y\partial_x$. Consider the isomorphism
		$$k[x^2,y^2,xy]\cong k[u,v,s]/(s^2-uv),\quad (x^2,y^2,xy)\mapsto (u,v,s).$$
Then $x\partial_x$ corresponds to $\psi=2u\partial_u+s\partial_s$, and $y\partial_x$ corresponds to $2s\partial_u+v\partial_s$. They generate an lc $1$-foliation $\sF=\sH/G$ on $Y$. 

By \autoref{thm:bir_sing_of_quotient} to be proved below, $Y/\sF$ is a klt affine singularity. Its coordinate ring is generated over $k$ by the sections $v$ and $u^ns^m$ where $n,m\geq 0$ are such that $2n+m=p$.
\end{example}

\section{Singularities of quotients}\label{section:sing_qt}

\subsection{Quotients by multiplicative derivations}\label{section:quotients_by_mult_der}
In this subsection we indicate a soft approach to singularities of quotients by multiplicative derivations.
So let $A$ be any $k$-algebra, and let $D\in\Der_k(A)$ be such that $D^{[p]}=uD$ where $u\in A^\times$. By \cite[Lemma 2.3]{Matsumoto_Purely_inseparable_coverings_of_RDP} we have $u\in A^D$. Notice that $u^{-1}\in A^D$ as well. We consider the finite extension of rings
		$$\varphi\colon A\hookrightarrow A'=A\left[ \lambda=u^{1/(1-p)} \right].$$
Since $u$ is invertible and $p-1$ is coprime with $p$, we see that $\varphi$ is finite \'{e}tale. In particular $D$ lifts uniquely to a $k$-derivation of $A'$, which we denote again by $D$. We have
		$$0=D(1)=D(\lambda^{p-1}u)=(p-1)u \lambda^{p-2}D(\lambda)$$
and so $D(\lambda)=0$. Moreover, by Hochschild's formula \autoref{eqn:Hochschild_formula} we have
		$$(\lambda D)^{[p]}=\lambda^p D^{[p]}=\lambda^{p}uD=\lambda D.$$
So by \autoref{prop:derivations_and_group_actions} the derivation $D'=\lambda D$ on $A'$ gives rise to a $\mu_p$-group action on $A'$, which is equivalent to a $\bZ/p$-grading
		$$A'=\bigoplus_{i=0}^{p-1}A'_i, \quad A_i'=\{s\in A'\mid D'(s)=is\}.$$
Since $D'(\lambda)=0$ we have $A'_0=A^D[\lambda]$, and thus a commutative diagram
		\begin{equation}\label{eqn:splitting_square}
		\begin{tikzcd}
		A^D[\lambda] \arrow[r, hook, "\oplus"] & A[\lambda]\\
		A^D\arrow[u, hook, "\varphi^D"] \arrow[r, hook] & A\arrow[u, hook, "\varphi"]
		\end{tikzcd}
		\end{equation}
Both vertical arrows $\varphi,\varphi^D$ are split, as they are $\mu_{p-1}$-cyclic \'{e}tale covers. Thus we obtain that \emph{the injection $A^D\hookrightarrow A$ splits as map of $A^D$-modules} (\footnote{This is also a consequence of the more general discussion given in \cite[\S 1]{Aramova_Reductive_derivations}.}).

\bigskip
We now derive some consequences of the above discussion. First we exploit the fact that, up to an \'{e}tale cover, a multiplicative derivation is given by a $\mu_p$-action. In particular, combining this observation with basic computations involving Jordan decomposition, we recover a well-known result about multiplicative quotients of regular local rings, see eg \cite[Theorem 2]{Rudakov_Shafarevich_Inseparable_morphisms_of_alg_surfaces}. 

\begin{proposition}\label{prop:normal_form_mult_derivations}
Suppose that $A$ a regular local $k$-algebra. If $D\in\Der_k(A)$ is singular and multiplicative then there exist formal coordinates $x_1,\dots,x_d\in \widehat{A}$ and $\lambda_1,\dots,\lambda_d\in\bF_p$ such that
		$$D=\sum_i\lambda_ix_i\frac{\partial}{\partial x_i}.$$
In particular, the completion of $A^D$ is normal toric.
\end{proposition}
The assumption of singularity is necessary, as shown by the multiplicative derivation $(1+x)\partial_x$ on $k[x]_{(x)}$.
\begin{proof}
Since $D$ is singular it must send $\fm_A^i$ to itself for every $i$, where $\fm_A$ is the maximal ideal of $A$. This implies that $D$ extends uniquely to a continuous multiplicative derivation of the completion of $A$. Since the vertical arrows in \autoref{eqn:splitting_square} are \'{e}tale and induce isomorphisms on residue fields, they induce isomorphisms of completions. So we may assume that $A$ is complete and that $D$ is given by a continuous $\mu_p$-action whose fixed locus contains the closed point. The action of $\mu_p$ on $A\cong k(A)\llbracket x_1,\dots,x_n\rrbracket$ can be linearized, see the proof of \cite[Corollary 1.8]{Satriano_Chevally_Todd_thm_for_lr_group_schemes}. Then it is given by the action of a single matrix $M\in\GL_n(A/\fm_A)$. We have $M^p=M$ and thus the minimal polynomial of $M$ divides $T^p-T$. So it cannot have multiple roots, and we deduce that $M$ is semi-simple. As $M^p=M$ we see that its eigenvalues are elements of $\bF_p$. So after an $A/\fm_A$-linear change of coordinates we may assume that $M$ is diagonal. Therefore the derivation $D$ can be given a normal form 
		\begin{equation*}
		D=\sum_i\lambda_ix_i\frac{\partial}{\partial x_i},\quad \lambda_i\in\bF_p.
		\end{equation*}
From this it is easily seen that $A^D$ is generated by monomials, and thus it is toric.
\end{proof}

\begin{remark}\label{rmk:toric_quotients_are_klt}
Quotients of regular local rings by multiplicative derivations are formally toric, hence klt. In general, they are not canonical: if we consider the action of $\partial_{a,b}$ on $\bA^2$ (see \autoref{example:sing_of_toric_derivation}), then the quotient singularity is canonical if and only if it is Gorenstein, which happens if and only if $(a,b)=(1,-1)$ or $(a,b)=(1,0)$ \cite[Remark 2.4.1]{Hirokado_Singularities_of_mult_derivations_and_Zariski_surfaces}.
\end{remark}

Next we use the splitting of the bottom arrows in \autoref{eqn:splitting_square} to descend some cohomological properties from $A$ to $A^D$. For the definitions of $F$-singularities appearing in the next theorem, see for example \cite{Ma_Polstra_F_singularities_commutative_algebra_approach}.

\begin{theorem}\label{thm:cohom_properties_of_mult_quotient}
Suppose that $A$ is Noetherian and $F$-finite, and that $D\in\Der_k(A)$ is multiplicative. Then:
	\begin{enumerate}
		\item If $A$ satisfies Serre's property $S_r$ then so does $A^D$.
		\item If $A$ is $F$-pure (resp. $F$-rational, $F$-injective, $F$-regular), so is $A^D$.
	\end{enumerate}
\end{theorem}
\begin{proof}
All these properties can be checked on localizations, and passing to the ring of constants commutes with localization at prime ideals. So we may assume that $A$ is local with maximal ideal $\fm$. Then $A^D$ is also local with maximal ideal $\fn=\fm^D$. By $F$-finiteness the map $A^D\hookrightarrow A$ is finite, so $\dim A=\dim A^D$.

Our main observation is the following one: since $\sqrt{\fn A}=\fm$, we have
		$$
		H^\bullet_\fm(A)=H^\bullet_{\fn A}(A)=H^\bullet_\fn(A)
		$$
where on the right-hand side we consider $A$ as an $A^D$-module \cite[Propositions 7.3, 7.15.(2)]{24_hours_of_cohomology}. As mentioned above, by \autoref{eqn:splitting_square} we can write $A=A^D\oplus B$ for some finite $A^D$-module $B$, and so the local cohomology splits accordingly, ie
	\begin{equation}\label{eqn:splitting_local_cohom}
	H^\bullet_\fm(A)=H^\bullet_\fn(A^D\oplus B)=H^\bullet_\fn(A^D)\oplus H^\bullet_\fn(B).
	\end{equation}
Let us show the first statement. Serre's property $S_r$ states that, after localizing at any prime, the depth along the maximal ideal is at least $\min\{\dim, r\}$. Since $\dim A=\dim A^D$, it suffices to show that the depth does not decrease upon passing to the sub-ring of constants. By \cite[Theorem 9.1]{24_hours_of_cohomology} we have
		$$\depth_\fm(A)=\inf_i\{ H^i_\fm(A)\neq 0\},\quad
		\depth_{\fn}(A^D)=\inf_i\{H^i_\fn(A^D)\neq 0\}.$$
So it follows from \autoref{eqn:splitting_local_cohom} that $\depth_\fn A^D\geq \depth_\fm A$ as desired.

Next we discuss descent of $F$-singularities. $F$-purity and $F$-regularity descend to split subrings \cite[Theorem 3.9 and Exerc. 9 p.13]{Ma_Polstra_F_singularities_commutative_algebra_approach}. $F$-injectivity and $F$-rationality do not in general (see 
\cite[Section 8]{Ma_Polstra_F_singularities_commutative_algebra_approach}
and \cite{Watanabe_F_rationality_and_counterexamples_to_Boutot}). Hopefully, the key fact is our situation is that $\dim A=\dim A^D$. We prove that $F$-rationality descends, the $F$-injective case is similar.

Suppose that $A$ is $F$-rational: this means that $A$ is Cohen--Macaulay and that given any $c\in A$ not contained in any minimal prime, there exists an $e>0$ such that the composition
			\begin{equation}\label{eqn:F_rational}
			H^d_{\fm}(A)\overset{f_A^e}{\to} H^d_{\fm}(F^e_*A)\overset{F^e_*c}{\longrightarrow} H^d_{\fm}(F^e_*A)
			\end{equation}
is injective, where $d=\dim A$ and $f_A^e=H^d_\fm(A\to F_*^eA)$. As seen above, $A^D$ is also Cohen--Macaulay. If we assume that $c\in A^D$ is not contained in any minimal prime, then by going-up $c$ is not contained in any minimal prime of $A$ either, and the above sequence is thus injective. Each local cohomology module in the sequence splits accordingly to \autoref{eqn:splitting_local_cohom}. Since $c\in A^D$, the action of $F_e^*c$ preserves each one of the summands. In other words, \autoref{eqn:F_rational} splits as
		$$\begin{tikzcd}
		H^d_{\fn}(A^D)\oplus H^d_\fn(B)  \arrow[r, "f_{A^D}^e\oplus f_B^e"] &
		 H^d_{\fn}(F^e_*(A^D))\oplus  H^d_{\fn}(F^e_*B) \arrow[r, "F^e_*c \oplus F^e_*c"] &
		 H^d_{\fn}(F^e_*(A^D))\oplus  H^d_{\fn}(F^e_*B).
		\end{tikzcd}$$
Since it is injective, the sequence given by the first summands is also injective. Thus $A^D$ is $F$-rational, as claimed.
\end{proof}

This has the following consequence for surface singularities. We say that a two-dimensional germ of surface over $k$ is a \textbf{linearly reductive quotient singularity} if $\widehat{\sO}$ is isomorphic to $k\llbracket x,y\rrbracket^G$ where $G$ is a linearly reductive group scheme acting freely away from the origin \cite[\S 6]{Liedtke_Martin_Matsumoto_Lrq_sing}.

\begin{corollary}
Suppose that $k$ is algebraically closed. Let $(\sO,\fm)$ be a two-dimensional linearly reductive quotient singularity over $k$, and $D\in\Der_k(\sO)$ be a multiplicative derivation. Then $\sO^D$ is a linearly reductive quotient singularity.
\end{corollary}
\begin{proof}
Recall that two-dimensional linearly reductive quotient singularities are the same as $F$-regular ones \cite[Theorem 5.11]{Liedtke_Martin_Matsumoto_Lrq_sing}. So $\sO$ is $F$-regular, and then $\sO^D$ is $F$-regular by \autoref{thm:cohom_properties_of_mult_quotient}.
\end{proof}

\begin{remark}
In the notations of the corollary, say that $\widehat{\sO}=k\llbracket x,y\rrbracket^G$. Then it is not known whether $\widehat{\sO}^D$ is the quotient of $k\llbracket x,y\rrbracket$ by an extension of $G$ and $\mu_p$. 

This is the case when $G$ is discrete (with order invertible in $k$). Indeed, as $\widehat{\bA}^2_{x,y}\to \Spec(\widehat{\sO})$ is \'{e}tale above the complement of the closed point, $D$ can be lifted to a $G$-invariant element of $\Gamma(\widehat{\bA}^2\setminus\{\bold{0}\},\Der_k^\text{cont}k\llbracket x,y\rrbracket)$. By reflexivity, $D$ extends to a regular element of $\Der_k^\text{cont}k\llbracket x,y\rrbracket$. It defines a continuous $\mu_p$-action on $\widehat{\bA}^2$ which commutes with the action of $G$. Then $\widehat{\sO}^D$ is the quotient of $k\llbracket x,y\rrbracket$ by $G\times\mu_p$.

However, if $G$ is not discrete then the question is much more complicated. See \cite[\S 8.2, especially Lemma 8.14]{Liedtke_Martin_Matsumoto_Torsors_over_RDP} and the references therein for further discussion.
\end{remark}

\subsection{Birational singularities of quotients}
In this section we study singularities of infinitesimal quotients in arbitrary dimensions from the point of view of birational geometry. It is convenient to do so for pairs and not only varieties, so we make the following definition.

\begin{definition}\label{def:quotient_of_pair}
Let $X$ be a normal $k$-scheme of finite type, $\Delta$ a $\bQ$-Weil divisor on $X$ and $\sF$ a $1$-foliation on $X$. Let $q\colon X\to X/\sF=Y$ be the quotient. We define on $Y$ the $\bQ$-Weil divisor 
	$$\Delta_Y=\sum_{E}\left(1-\epsilon_\sF(E)\frac{p-1}{p} \right)\coeff_{E}(\Delta)\cdot q(E)$$
where $E$ runs through the prime divisors of $\Supp(\Delta)$.
\end{definition}

The extra factors are thrown in to accommodate the adjunction formula along $q$ (that is, \autoref{prop:adjunction_formula}) in presence of an extra divisor. Indeed, we have:

\begin{lemma}[cf {\cite[Proposition 1]{Rudakov_Shafarevich_Inseparable_morphisms_of_alg_surfaces}}]\label{lemma:pullback_quotient_divisors}
Let $X$ be a normal $k$-scheme of finite type and $\sF$ be a $1$-foliation on $X$. Let $q\colon X\to Y$ be the quotient. For a prime divisor $E\subset X$ with image $q(E)=E^Y\subset Y$:
	\begin{enumerate}
		\item if $E$ is $\sF$-invariant then $q^*E^Y=E$;
		\item if $E$ is not $\sF$-invariant then $q^*E^Y=pE$.
	\end{enumerate}
\end{lemma}
\begin{proof}
We can work on an \'{e}tale neighbourhood of the generic point of $E$, where $X$ and $E$ are regular and $\sF$ a sub-bundle of the tangent sheaf. Then by \autoref{example:reg_foliations_on_reg_varieties} we may assume that $X=\bA^n_{x_1,\dots,x_n}$, that $\sF$ is generated by $\frac{\partial}{\partial x_1},\dots, \frac{\partial}{\partial x_r}$ with $r<n$, and that $E$ is cut out by a linear polynomial. The invariant subring, whose spectrum gives $Y$, is given by $k[x_1^p,\dots,x_r^p,x_{r+1},\dots,x_r]$.
	\begin{itemize}
		\item If the polynomial is not a linear combination of $x_1,\dots,x_r$, we make a change of coordinate and assume that $E$ is given by $(x_{r+1}=0)$. Then $E$ is $\sF$-invariant and $E_Y=(x_{r+1}=0)$ so $q^*E^Y=E$.
		\item If the polynomial is cut out by the $x_1,\dots,x_r$, we may similarly assume that $E=(x_r=0)$. Then $E$ is not $\sF$-invariant, as $\frac{\partial}{\partial x_r}|_E\notin T_E$, and $E^Y=(x_r^p=0)$ so $q^*E^Y=pE$.
	\end{itemize}
This completes the proof.
\end{proof}

\begin{proposition}[Log adjunction formula]\label{prop:log_adjunction}
Notations as in \autoref{def:quotient_of_pair}. Then we have an equality of $\bQ$-Weil divisor classes
		$$q^*(K_Y+\Delta_Y)=K_X+\Delta+(p-1)K_\sF.$$
\end{proposition}
\begin{proof}
Let $E$ be a prime divisor on $\Supp(\Delta)$. Then by \autoref{lemma:pullback_quotient_divisors} we have in any case
		$$q^*\left(1-\epsilon(E)\frac{p-1}{p}\right)q(E)=E$$
and so $q^*\Delta_Y=\Delta$. Combining this equality with the adjunction formula \autoref{prop:adjunction_formula} yields the result.	
\end{proof}

We also note that taking quotients preserves $\bQ$-Gorenstein properties:

\begin{lemma}\label{lemma:quotient_is_Q_Gor}
Let $X$ be a normal $k$-scheme of finite type, $\sF$ be a $1$-foliation on $X$ with quotient $q\colon X\to Y$. Let $\Delta$ be a $\bQ$-Weil divisor on $X$. If $K_X+\Delta$ and $K_\sF$ are $\bQ$-Cartier, then $K_Y+\Delta_Y$ is also $\bQ$-Cartier.
\end{lemma}
\begin{proof}
Since $Y$ is normal, it has a well-defined canonical divisor $K_Y$ which is invertible on a big open subset $U$. Over that locus, the log adjunction reads $q^*(K_Y+\Delta_Y)|_U=(K_X+\Delta+(p-1)K_\sF)|_{q^{-1}U}$. Now $K_X+\Delta$ and $K_\sF$ are by assumption $\bQ$-Cartier, so for $n>0$ big enough the Weil divisor $n(K_X+\Delta+(p-1)K_\sF)$ is Cartier. Thus its pullback on $Y^{(-1)}$ is Cartier as well. Since it is represented on the big open subset $U^{(-1)}$ by the divisor $np(K_{Y^{(-1)}}+\Delta_{Y^{(-1)}})$, it follows that $K_{Y^{(-1)}}+\Delta_{Y^{(-1)}}$ is $\bQ$-Cartier, and therefore $K_Y+\Delta_Y$ is $\bQ$-Cartier as claimed.
\end{proof}

The main theorem of this section reads as follows. 

\begin{theorem}\label{thm:bir_sing_of_quotient}
Let $(X,\Delta)$ be a normal pair and $\sF$ be a $\bQ$-Gorenstein $1$-foliation on $X$. Let $q\colon X\to X/\sF=Y$ be the quotient morphism and $\Delta_Y$ be the divisor on $Y$ induced by $\Delta$ as in \autoref{def:quotient_of_pair} above.
	\begin{enumerate}
		\item Assume that $\sF$ is canonical. Then if $(X,\Delta)$ is terminal (resp. canonical, klt, lc), so is $(Y,\Delta_Y)$.
		\item Assume that $\sF$ is klt. If $(X,\Delta)$ is lc and $\lfloor\Delta^{\sF\text{-inv}}\rfloor =0$, then $(Y,\Delta_Y)$ is klt.
		\item Assume that $\sF$ is lc. Then:
			\begin{enumerate}
				\item If $(X,\Delta)$ has at worst klt singularities, so does $(Y,\Delta_Y)$;
				\item If $(X,\Delta)$ is lc, so is $(Y,\Delta_Y)$.
			\end{enumerate}
	\end{enumerate}
\end{theorem}

\begin{remarks}\label{rmk:bir_sing_qt}
	\begin{enumerate}
		\item As the proof will show, in concrete cases a finer analysis might be possible (see \autoref{prop:discrep_along_min_resol} for some examples).
		\item Assume that $\sF$ is strictly log canonical. Then even if $X$ is smooth, the singularities of $Y$ need not be milder than klt, as we saw in \autoref{rmk:toric_quotients_are_klt}.
		\item If both $X$ and $Y$ are regular, then $\sF$ is regular according to \autoref{lemma:reg_quotient_implies_reg_foliation}. But in general the singularities of $\sF$ cannot be quantified from the singularities of $X$ and $Y$: in \autoref{example:ring_of_csts} we have seen the canonical singularity $D^0_{4}$ $(p=2)$ arising as the quotient of $k\llbracket x,y\rrbracket$ by the non-lc derivation $x^2\partial_x+y^2\partial_y$. See \cite[Proposition 2.3]{Liedtke_Uniruled_surfaces_of_gen_type} for more examples with $p=2$. The surface case is nonetheless special, and we will see in \autoref{thm:lc_foliation_surfaces} below that what is at play here is that the $D^0_4$ singularity is not $F$-pure.
		\item Even if $\sF$ and $X/\sF$ are midly singular, $X$ need not be so. Indeed, consider the affine scheme $X=\Spec(k[\bold{x},z]/(z^p-s(\bold{x}))$ and the $1$-foliation $\sF$ generated by $\partial/\partial z$. As observed in \cite[Example 2.14]{Bernasconi_Counterexample_MMP_for_foliations_in_pos_char}, $\sF$ is canonical. The quotient $X/\sF$ is the affine variety with coordinate ring $k[\bold{x},z^p]/(z^p-s(\bold{x}))\cong k[\bold{x}]$, hence it is regular. But we can choose $s(\bold{x})$ such that $X$ is normal but not lc: for example $s(\bold{x})=x_1^n+x_2^m$ with $n,m\gg 1$ not divisible by $p$.
	\end{enumerate}
\end{remarks}

\begin{proof}
By \autoref{lemma:quotient_is_Q_Gor} the $\bQ$-Weil log canonical divisor $K_Y+\Delta_Y$ is $\bQ$-Cartier, and thus we may investigate its birational singularities. Consider a proper birational morphism $\mu\colon Y'\to Y$, and write $K_{Y'}+\mu_*^{-1}\Delta_{Y}=\mu^*(K_Y+\Delta_Y)+\sum a_{E}E$ where $E$ runs through the exceptional prime divisors of $\mu$. We are interested in the numbers $a_E$. To compute them we look at the following commutative diagram
	$$\begin{tikzcd}
	X\arrow[d, "q"] & X'\arrow[l, "\nu"]\arrow[d, "q'"] \\
	Y & Y'\arrow[l, "\mu"]
	\end{tikzcd}$$
where $X'$ is the normalization of $Y'$ in $K(X)$. The morphism $\nu\colon X'\to X$ indeed exists and is uniquely determined by Zariski's Main Theorem. Notice that $q'$ is an homeomorphism; in particular $q'$ induces a bijection between the $\mu$-exceptional prime divisors and the $\nu$-exceptional ones. So if $E$ is $\mu$-exceptional, let us write $E'$ the corresponding $\nu$-exceptional divisor. Let also $\sF'=\nu^*\sF$ be the foliation induced by $\sF$ on $X'$, so that $q'$ is the quotient morphism.
%, we have $(q')^*E=E'$ if and only if $E'$ is $\sF'$-invariant; and if $E'$ is not $\sF'$-invariant then $(q')^*E=pE'$ \textbf{reference}. 
We now write down the pullback formulas for every canonical divisor in sight:
		\begin{eqnarray*}
		K_{Y'}+\mu_*^{-1}\Delta_Y &=& \mu^*(K_Y+\Delta_Y)+\sum a_EE, \\
		K_X+\Delta &=& q^*(K_Y+\Delta_Y)+(1-p)K_{\sF}, \\
		K_{X'}+\nu^{-1}_*\Delta &=&(q')^*(K_{Y'}+\mu_*^{-1}\Delta_Y)+(1-p)K_{\sF'},\\
		K_{\sF'}&=&\nu^*K_{\sF}+\sum b_{E'}E', \\
		K_{X'}+\nu_*^{-1}\Delta &=& \nu^*(K_X+\Delta)+\sum c_{E'}E'.
		\end{eqnarray*}
Here we use the notations of \autoref{def:quotient_of_pair}, and the third equality follows from \autoref{prop:log_adjunction} and the observation that $(\nu^{-1}_*\Delta)_{Y'}=\mu^{-1}_*\Delta_Y$ (which can be verified over the generic points of $\mu^{-1}\Delta_Y$, where $\mu$ is an isomorphism). To simplify the notations, we may localize at the generic point of some $E$, and assume that there is a unique exceptional divisor. If we apply $(q')^*$ to the very first equation in the above list, we find
		$$a_E\delta(E') E'=(q')^*(K_{Y'}+\mu^{-1}_*\Delta_Y)-(q')^*\mu^*(K_Y+\Delta_Y)$$
where
		$$\delta(E')=\begin{cases}
		1 & \text{if }E'\text{ is }\sF'\text{-invariant}, \\
		p & \text{otherwise}.
		\end{cases}$$
Using the canonical isomorphism $(q')^*\mu^*\cong \nu^*q^*$ and substituting in the other formulae on the right-hand side, we find (\footnote{
	Let us highlight a subtle point here. The adjunction formula \autoref{prop:log_adjunction} is an equality of divisor classes, not of actual divisors. What yields actual $\bQ$-Weil divisors is the comparison between the class $K_X$, (resp.\ $K_Y$, resp.\ $K_\sF$) with $K_{X'}$ (resp.\ $K_{Y'}$, resp.\ $K_{\sF'}$. The reader will check that we obtain the claimed formula using only such comparisons.
}) 
the equation
		$$a_E\delta(E') = c_{E'}+(p-1)b_{E'}.$$
In other words,
		\begin{equation}\label{eqn:discrepancies_of_quotient}
		\begin{cases}
		\text{if }E'\text{ is }\sF'\text{-invariant}:&
			a_E=c_{E'}+(p-1)b_{E'}\\
		\text{if }E'\text{ is not }\sF'\text{-invariant}:&
			a_E=\frac{1}{p}\left(c_{E'}+(p-1)b_{E'}\right).
		\end{cases}
		\end{equation}
The number $c_{E'}$ has a lower bound according to the birational singularities of $(X,\Delta)$. The number $b_{E'}$ has a lower bound according to the birational singularities of the foliation $\sF$ (and in the log canonical and klt cases, it might depend on the $\sF'$-invariance of $E'$; in particular, if $b_{E'}$ is negative then $E'$ is not $\sF'$-invariant). A simple case-by-case analysis based on \autoref{eqn:discrepancies_of_quotient} concludes the proof.
\end{proof}

We indicate several corollaries of \autoref{thm:bir_sing_of_quotient}.

\begin{corollary}\label{prop:can_foliations_on_surfaces}
Let $X$ be a regular variety and $\sF$ a $1$-foliation on $X$. Then:
	\begin{enumerate}
		\item if $\sF$ is canonical then it is regular outside a closed subset of codimension $\geq 3$.
		\item if $\sF\subsetneq T_{X/k}$ then $\sF$ cannot be terminal.
	\end{enumerate}
\end{corollary}
\begin{proof}
By \cite[2.30]{Kollar_Singularities_of_the_minimal_model_program} a terminal variety is regular in codimension $2$. So if $\sF$ is canonical, then by \autoref{thm:bir_sing_of_quotient} the quotient $X/\sF$ is terminal, hence regular in codimension $2$. Thus $\sF$ is regular in codimension $2$ by \autoref{lemma:reg_quotient_implies_reg_foliation}.

Now if $\sF\subsetneq T_{X/k}$ was terminal then it would be generically regular by the previous paragraph, but by \autoref{example:reg_foliations_not_terminal} regular foliations on regular varieties are not terminal: contradiction.
\end{proof}

As a consequence, we see that on regular surfaces in positive characteristic, there is no terminal $1$-foliations and that the only canonical ones are the regular ones. This is in sharp contrast with the characteristic zero $0$ case, where there is a larger supply of canonical foliations \cite[III.i.3]{McQuillan_Panazzolo_Almost_etale_resolution_of_foliations}.

If we consider singular underlying surfaces, then the discrepancies of $\sF$ along the minimal resolution are usually non-negative:
\begin{proposition}\label{prop:discrep_along_min_resol}
Let $S$ be a normal surface and $\sF$ be a $\bQ$-Gorenstein $1$-foliation on $S$. Assume that $S/\sF$ has canonical singularities (resp. is regular). Then $a(E;\sF)\geq 0$  (resp. $a(E;\sF)>0$) for every exceptional divisor $E$ on the minimal resolution of $S$.
\end{proposition}
\begin{proof}
Let $T=S/\sF$, let $\pi\colon S'\to S$ be the minimal resolution of $S$, and let $T'$ be quotient of $S'$ by $\pi^*\sF$. Then we have a commutative diagram
		$$\begin{tikzcd}
		S\arrow[d] & S'\arrow[l, "\pi"]\arrow[d, "q"] \\
		T & T'.\arrow[l]
		\end{tikzcd}$$
Let $E\subset S'$ be a prime $\pi$-exceptional divisor, with image $E'=q(E)\subset T'$. Then as in the proof of \autoref{thm:bir_sing_of_quotient}, we find
		$$a(E';T)\delta(E)=a(E;S')+(p-1)a(E;\sF).$$
First assume that $T$ is canonical. As $\delta(E)\in \{1,p\}$, the left-hand side of the equality is non-negative. Since $S'$ is the minimal resolution of $S$, we have $a(E;S')\leq 0$. Thus we must have $a(E;\sF)\geq 0$. If $T$ is regular, then the left-hand side is positive and so $a(E;\sF)>0$.
\end{proof}

\begin{example}\label{example:resolution_toric_sing}
Let $T=(\bold{0}\in \bA^2_{x,y})$ and $\sG$ be the $1$-foliation generated by $\partial_{a,-b}$ (\autoref{example:toric_derivation}). Then $S=T/\sG$ is klt (even canonical if $a=-b$). Let $\sF$ be the unique $1$-foliation such that $T^{(-1)}=S/\sF$. Then $\sF$ is $\bQ$-Gorenstein since $S$ is $\bQ$-factorial \cite[Corollary 4.11]{Tanaka_MMP_for_excellent_surfaces}. By the above proposition, if $\pi\colon S'\to S$ is the minimal resolution then $K_{\pi^*\sF}=K_\sF+E$ where $E\geq 0$ and $\Supp(E)=\Exc(\pi)$.
\end{example}

\begin{corollary}\label{cor:Reid_lemma_insep_map}
Let $X$ be an lc (resp. klt) $k$-scheme of finite type and $s\in H^0(X,\sO_X)$. If the $1$-foliation $\Ann(s)$ defined by 
	$$\Ann(s)(U)=\{D\in T_{X/k}(U)\mid D(s_U)=0\}, \quad U\subset X\text{ open},$$
is lc and properly contained in $T_{X/k}$, then the normalized $p$-cyclic cover $\left( X[\sqrt[p]{s}]\right)^\nu$ is lc (resp. klt).
\end{corollary}
\begin{proof}
It is easy to check that $\Ann(s)$ is indeed a $1$-foliation, and that $\Ann(s)\neq T_{X/k}$ if and only if $s$ is not a $p$-th power. If $Y$ is the normalization of $X[\sqrt[p]{s}]$ then $X/\Ann(s)=Y^{(-1)}$. As $Y$ and $Y^{(-1)}$ are abstractly isomorphic, the result follows from \autoref{thm:bir_sing_of_quotient}.
\end{proof}

\begin{remarks}
	\begin{enumerate}
		\item \autoref{cor:Reid_lemma_insep_map} can be generalized to more general $p$-cyclic coverings as follows. Let $\sL$ be a line bundle on $X$, and $s\in H^0(X,\sL^{-p})$ be a section that does not have a $p$-th root in $H^0(X,\sL^{-1})$. Then we use $s\colon \sL^p\to \sO_X$ to give an $\sO_X$-algebra structure to the direct sum $\bigoplus_{i=0}^{p-1}\sL^i$. Taking its normalized relative spectrum over $X$ yields a finite purely inseparable $Y\to X$. Let $\{U_\alpha\}$ be a affine cover that trivializes $\sL$. For each $U_\alpha$, choose a generator $\sigma_\alpha\in \sL(U_\alpha)$: we can write $s|_{U_\alpha}=u_\alpha\sigma_{\alpha}^{-p}$ in $\sL^{-p}(U_\alpha)$ for some $u_\alpha\in\sO_X(U_\alpha)$. This element $u_\alpha$ does not depend on $\sigma_\alpha$, up to scaling by an element of $\sO_X(U_\alpha)^p$. Thus the assignments
				$$U_\alpha\mapsto \{D\in T_{X/k}(U_\alpha)\mid D(u_\alpha)=0\}$$
	define a $1$-foliation $\Ann(s)$ on $X$. Since 
			$$Y_{U_\alpha}\cong \Spec \left(\sO_X(U_\alpha)[T]/(T^p-u_\alpha)\right)^\nu$$
	we may apply \autoref{cor:Reid_lemma_insep_map} to obtain some singularity restrictions on $Y$ as soon as $\Ann(s)$ is lc.
		\item If $X$ is a regular variety, it is easy to find local generators of $\Ann(s)$, at least up to saturation. Let $x_1,\dots,x_n$ be a regular system of parameters of $\sO_{X,z}$. By \autoref{lemma:descent_derivation_from_completion} there exist $D_i\in\Der_k(\sO_{X,z})$ such that $D_i(x_j)=\delta_{ij}$. Write $s_i=D_i(s)$: since $s$ is not a $p$-th power, we may assume that $s_1\neq 0$. Then it is immediate that
				$$s_{i}D_1-s_{1}D_i\in\Ann(s) \quad \forall i=2,\dots,n.$$
	Now $\Ann(s)$ has corank $1$, so these $n-1$ derivations generate $\Ann(s)$ up to saturation as they are generically linearly independent. 
	\item For example, let $(s\in S)$ be a germ of regular surface and take $\varphi\in H^0(S,\sO_S)$. Choose a regular system of parameters $x,y\in \sO_{S,s}$ and let $D_x,D_y\in \Der_k(\sO_{S,s})$ be the derivations afforded by \autoref{lemma:descent_derivation_from_completion}. Then $\Ann(\varphi)$ is generated up to saturation by $D_y(\varphi)D_x-D_x(\varphi)D_y$. Expanding $\varphi=\sum_{i,j}\varphi_{ij}x^iy^j$ in $\widehat{\sO}_{S,s}$, we have
		$$D_y(\varphi)\equiv\varphi_{01}+\varphi_{11}x+2\varphi_{02}y \quad\text{and}\quad
		D_x(\varphi)\equiv \varphi_{10}+2\varphi_{20}x+\varphi_{11}y$$
where the equalities are taken module $\fm^2$. Assuming that $D_y(\varphi)$ and $D_x(\varphi)$ have trivial greatest common divisor, we see:
	\begin{enumerate}
		\item $\Ann(\varphi)$ is regular at $s$ if and only if $d\varphi=D_x(\varphi)dx+D_y(\varphi)dy\notin \fm\Omega^1_{S/k}$;
		\item $\Ann(\varphi)$ is strictly lc at $s$ if and only if: $\varphi_{01}=0=\varphi_{10}$ and the matrix
				$$\begin{pmatrix}
				\varphi_{11} & -2\varphi_{20} \\
				2\varphi_{02} & -\varphi_{11}
				\end{pmatrix}$$
		is non-nilpotent (that is, its determinant and its trace are not both zero).
	\end{enumerate}
	\item On surfaces, singularities of normalized $p$-cyclic covers can be handled with other methods: for example, see \cite{Kawamata_Index_1_covers_of_klt_surfaces, Arima_Index_1_covers_of_surfaces} for the case of canonical index-one cover of klt surfaces.
	\end{enumerate}
\end{remarks}

\subsection{Singularities of surface quotients}
In this section we consider more specifically the singularities of quotients of regular surfaces. From \autoref{thm:bir_sing_of_quotient} we get:

\begin{corollary}\label{cor:quotient_has_rat_sing}
Let $(s\in S)$ be a normal surface germ and $\sF$ a log canonical $1$-foliation on $S$. Assume either that
	\begin{enumerate}
		\item $S$ is klt, or
		\item there exists a non-zero $\bQ$-Weil divisor $\Delta$ such that $(S,\Delta)$ is lc.
	\end{enumerate}
Then the germ $(s\in S/\sF)$ has rational singularities.
\end{corollary}
\begin{proof}
Let $T=S/\sF$ be the quotient. In the first case, $T$ is also klt by the theorem and klt surfaces are rational \cite[Fact 3.4]{Tanaka_X_method_for_surfaces}. In the second case, if $\Delta_T$ is defined as in \autoref{def:quotient_of_pair} then $(T,\Delta_T)$ is lc by the theorem. As $\Delta_T\neq 0$ we deduce from \cite[Proposition 2.28]{Kollar_Singularities_of_the_minimal_model_program} that $T$ has rational singularities.
\end{proof}

\begin{remark}
\autoref{cor:quotient_has_rat_sing} may fail if $\sF$ is not log canonical: see \cite[Proposition 2.3]{Liedtke_Uniruled_surfaces_of_gen_type} for an example where the quotient has an elliptic singularity. It would be interesting to determine what may happen when $S$ is an elliptic singularity, and $\sF$ a log canonical $1$-foliation.
\end{remark}

For $1$-foliations of rank one, we have a partial converse of \autoref{thm:bir_sing_of_quotient}, which applies in particular to non-trivial $1$-foliations on regular surfaces:

\begin{theorem}\label{thm:lc_foliation_surfaces}
Let $X$ be a regular variety over $k$, and $\sF$ be a $1$-foliation of rank one on $X$. Then the following are equivalent:
	\begin{enumerate}
		\item $X/\sF$ is $F$-regular,
		\item $X/\sF$ is $F$-pure,
		\item $\sF$ is lc.
	\end{enumerate}
\end{theorem}
\begin{proof}
We may assume that $k$ is algebraically closed and that $X=\Spec(\sO)$ is the spectrum of a complete regular local two-dimensional ring with residue field $k$, and that $\sF$ is not regular. 

Suppose that $\sF$ is lc. By \autoref{cor:lc_implies_mult_sing} the $1$-foliation $\sF$ has multiplicative singularities. Then by $X/\sF$ is $F$-regular by \autoref{thm:cohom_properties_of_mult_quotient}.

Since $F$-regularity implies $F$-purity, we only need to show that if $A=\sO^\sF$ is $F$-pure then $\sF$ is lc. We follow the argument of \cite[Proposition 2.4]{Hara_Sawada_Splitting_of_Frobenius_sandwich}. The sheaf $\sF$ is free of rank $1$, so pick a generator $\partial$. We have $\partial^{[p]}=\alpha \partial$ for some $\alpha\in \sO$. We will show that $\partial$ is multiplicative, which will conclude by \autoref{cor:lc_implies_mult_sing}.

The inclusion $A\hookrightarrow A^{1/p}$ splits as map of $A$-modules. Since it factors through $\sO$, by restriction we get an $A$-module map $\varphi\colon \sO\to A$ splitting the natural inclusion. Over the regular locus of $X/\sF$ the sheaf $\sHom_{\sO_{X/\sF}}(\sO_X,\sO_X)$ is generated by $\partial$ over $\sO_X$ \cite{Yuan_Inseparable_Galois_theory}. This regular locus is big and the Hom sheaf is reflexive, thus $\partial$ generates over $\sO$ the endomorphism ring $\Hom_A(\sO,\sO)$. So we can write $\varphi=\sum_{i=0}^{p-1}a_i\partial^{[i]}$ for some $a_i\in \sO$. Since $\varphi$ splits $A=\sO^\partial\hookrightarrow\sO$, we see that $a_0=1$ and $\partial\circ \varphi=0$. As the $\partial^{[i]}$ are $\sO$-linearly independent \cite[Theorem 25.4]{Matsumura_Commutative_Ring_Theory}, by considering the $i=1$ term of $\partial\circ\varphi$ we find
		$$\partial(a_1)+1+\alpha a_{p-1}=0.$$
Now $\partial(a_1)$ belongs to the maximal ideal of $\sO$ (because $\sF$ is assumed to be singular, cf \autoref{lemma:singularity_Gor_foliation}). Therefore $\alpha$ must be invertible, and so $\partial$ is multiplicative.
\end{proof}

\section{Families of foliations}\label{section:variation}

We introduce a notion of parametrized families of foliations. We fix once and for all a perfect field $k$ of characteristic $p>0$.

\subsection{Definitions and universal families}

\begin{definition}\label{def:families_of_foliations}
Let $S$ be a locally Noetherian $F$-finite $k$-scheme. A \textbf{family of foliations of rank $r$} over $S$ is the data of
	\begin{itemize}
		\item a flat finite type $k$-morphism $f\colon \sX\to S$,
		\item two coherent $\sO_\sX$-modules $\sF\hookrightarrow \fT_{\sX/S}$, 
	\end{itemize}
subject to the following properties:
	\begin{enumerate}
		\item the fibers of $\sX\to S$ are geometrically normal and $f_*\sO_\sX=\sO_S$ holds,
		\item $\fT_{\sX/S}$ is flat over $S$, and its restriction to $\sX_s$ can naturally be identified with the tangent sheaf $T_{\sX_s/k(s)}$ for every $s\in S$,
		\item the quotient sheaf $\sQ=\fT_{\sX/S}/\sF$ is flat over $S$, and
		\item for any $s\in S$, the fiber $\sF_s\hookrightarrow T_{\sX_s/k(s)}$ is a foliation of rank $r$.
	\end{enumerate}
We define analogously \textbf{families of $1$-foliations of rank $r$ on $\sX$} over $S$.
\end{definition}

What is implicit in the definition is that, since both $\fT_{\sX/S}$ and $\sQ$ are flat over $S$, the sheaf $\sF$ is also flat over $S$ and so for every $s\in S$ we have an exact sequence
		\begin{equation}\label{eqn:restriction_of_exact_seq}
		0\to \sF_s\to T_{\sX_s/k(s)}\to \sQ_s\to 0
		\end{equation}
where we write $\sF_s=\sF\otimes k(s)$ and $\sQ_s=\sQ\otimes k(s)$.

\begin{remark}\label{rmk:technical_limitations}
In practice, a family of foliations $\sF\subset \fT_{\sX/S}$ is interesting if we can interpret $\sF$ itself as a collection of vector fields over $S$. So typically we want $\fT_{\sX/S}=T_{\sX/S}$, \emph{and this is the only case we will consider}. Given an arbitrary $\sX\to S$, for every $s\in S$ there is a natural morphism $T_{\sX/S}\otimes k(s)\to T_{\sX_s/k(s)}$
but in general it is not an isomorphism. In a few special cases it is, however:
	\begin{enumerate}
		\item When the morphism $\sX\to S$ is smooth, because then $\Omega^1_{\sX/S}$ is locally free. %More generally, a local computation shows that $T_{\sX/S}\otimes k(s)\to T_{\sX_s/k(s)}$ is an isomorphism if $\sX\to S$ is a semi-stable morphism.
		\item When $S$ is regular, and $T_{\sX/S}$ is $S_{2+\dim S}$ and flat over $S$. (\footnote{I am not aware of simple conditions, besides relative smoothness, that guarantee that $T_{\sX/S}$ satisfies Serre's condition $S_r$ beyond $r= 2$, even in the absolute case where $S$ is the spectrum of a field. If $X$ is Gorenstein and $\Omega_{X/k}^1$ is Cohen--Macaulay, then by duality \cite[Proposition 3.3.3]{Bruns_Herzog_CM_rings} its dual $T_{X/k}$ is also Cohen--Macaulay. The condition that $\Omega^1_{X/k}$ is Cohen--Macaulay appears to be quite restrictive: in case $X$ is lci, $\Omega_{X/k}^1$ satisfies $S_r$ if and only if $X$ is regular in codimension $r$ \cite[Proposition 9.7]{Kunz_Kahler_differentials}.})
		\item When $T_{\sX/S}$ is its own \emph{universal hull}, in the sense of \cite[\S 9.4]{Kollar_Families_of_varieties_of_general_type}.
	\end{enumerate}
\end{remark}

%In case $\sX\to S$ is a trivial family, we specialize \autoref{def:families_of_foliations} as follows.

%\begin{definition}\label{def:families_of_foliations_on_fixed_var}
%Let $X$ be a normal $k$-variety and $S$ a normal locally Noetherian $k$-scheme. A \textbf{family of foliations of rank $r$ on $X$} over $S$ is a family of foliations $(X_S\to S, \sF\hookrightarrow T^1_{X_S/S}))$ in the sense of \autoref{def:families_of_foliations}. We define families of $1$-foliations on $X$ in the same way.
%\end{definition}

%\begin{remarks}
%It is implicit in \autoref{def:families_of_foliations_on_fixed_var} that the fibers of $T^1_{X_S/S}$ can be identified with $T^1_{X_{k(s)}/k(s)}$. This is certainly true if $X$ is regular, because then $\Omega_{X/k}^1$ is locally free and since pullbacks commute with formation of duals for locally free sheaves we obtain that $(T^1_{X/k})_S=T^1_{X_S/S}$ in that case (for any $S$). Now suppose that $X$ is only normal, and that $S$ is normal as well. Then $X_S$ is normal, the two sheaves $(T^1_{X/k})_S$ and $T^1_{X_S/S}$ are $S_2$ (for the former, this follows from \cite[6.7.3]{EGA_IV.2}) and, by the regular case, they can be naturally identified on $\Reg(X)_S$. That locus is big on $X_S$, so it follows that $(T^1_{X/k})_S=T^1_{X_S/S}$ everywhere.
%\end{remarks}

\begin{lemma}\label{lemma:family_of_foliation_is_a_foliation}
Let $ \sF\hookrightarrow T_{\sX/S}$ be a family of foliations (resp. of $1$-foliations) of rank $r$. Assume that $S$ is $S_2$. Then $\sF$ is a foliation (resp. a $1$-foliation) of rank $r$. 
\end{lemma}
\begin{proof}
We need to prove that $\sF$ is saturated in $T_{\sX/S}$, since its generic fiber is already closed under Lie brackets (resp. under $p$-th powers), see \autoref{rmk:Lie_and_pth_cond_at_generic_point}. Now $\sF$ is flat over $S$ and its restriction to any fiber is $S_2$. As $S$ is $S_2$, if follows from \cite[6.4.1.(ii)]{EGA_IV.2} that $\sF$ is an $S_2$ $\sO_{\sX}$-module. It follows that the quotient sheaf $\sQ$ is $S_1$ (see \cite[2.60]{Kollar_Singularities_of_the_minimal_model_program}). Since $\sQ$ has full support, it is torsion-free \cite[0AUV]{Stacks_Project}. Hence $\sF$ is saturated in $T_{\sX/S}$.
\end{proof}

Non-trivial families of foliations (resp. of $1$-foliations) on constant smooth families exist. In fact, assuming the fiber to be is projective, there are even universal such families. This is a straightforward Quot scheme argument, which we spell out for completeness. We refer to \cite{FGA_explained} for the definition and the construction of the Quot scheme.

\begin{lemma}\label{lemma:iso_is_closed_cond}
Let $T\to S$ be a proper morphism of Noetherian schemes and $u\colon \sM\to \sN$ a surjective morphism of coherent $\sO_T$-modules. Assume that $\sM$ is flat over $S$. Then there is a closed subscheme $S'$ of $S$ such that: a morphism $V\to S$ factorizes through $S'$ if and only if $u_V$ is an isomorphism.
\end{lemma}
\begin{proof}
This is \cite[Chapter 2, \S 4, Proposition]{Raynaud_Flat_modules_in_AG}.
\end{proof}

\begin{proposition}\label{prop:universal_families_of_foliations}
Let $X$ be a regular projective connected $k$-scheme $X$ and $\dim X>r\geq 1$ be an integer. Then there is a locally closed subscheme $Q_0(r)$ (resp. a locally closed subscheme $Q_1(r)$) of $\Quot_{T_{X/k}/X/k}$ that parametrizes families of foliations (resp. of $1$-foliations) of rank $r$ on $X$ over Noetherian bases.
\end{proposition}
\begin{proof}
Write $T=T_{X/k}$. Notice that since $X$ is regular, for every $k$-scheme $S$ the sheaf $T_{X_S/S}$ is the pullback of $T$ through the projection $X_S=X\times_kS\to X$. Fix an ample line bundle $L$ on $X$. The Quot scheme $\Quot_{T/X/k}$ is the disjoint union of the projective schemes $Q(\varphi)=\Quot_{T/X/k}^{L,\varphi}$, where $\varphi\in\bQ[t]$. Each product $X_{Q(\varphi)}$ supports a universal quotient
		$T_{X_{Q(\varphi)}/Q(\varphi)}\twoheadrightarrow \sQ^{\varphi,\text{univ}}$.
Let $\sF^{\varphi,\text{univ}}$ be its kernel. Note that the relative rank of $\sF^{\varphi,\text{univ}}$ over $Q(\varphi)$ is uniquely determined by the Hilbert polynomial $\varphi$. 

We have to show that there exist a locally closed subscheme $Q_0(r,\varphi)$ (resp. $Q_1(r,\varphi)$) of $Q(\varphi)$ such that $S\to Q(\varphi)$ factors through $Q_0(r,\varphi)$ (resp. through $Q_1(r,\varphi)$) if and only if the pullback of the universal quotient through $X_S\to X_Q$ is a family of foliations (resp. of $1$-foliations) on $X$. Then we will have $Q_i(r)=\bigsqcup_\varphi Q_i(r,\varphi)$. 

For simplicity, from now on we drop the $\varphi$ from the notation. For every point $q\in Q$, by flatness of $\sQ^\text{univ}$ we have an exact sequence of $\sO_{X_{k(q)}}$-coherent modules:
		$$0\to \sF^{\text{univ}}_q\to T_{X_{k(q)}/k(q)}\to \sQ^{\text{univ}}_q\to 0.$$
We define a few subschemes of $Q$ as follows.
	\begin{enumerate}
		\item \emph{There is an open subset $\sU$ parametrizing saturated and reflexive $\sF^{\text{univ}}_q$.} The fibers $\sQ^{\text{univ}}_q$ have all full support, so they are torsion-free if and only if if they are $S_1$ \cite[0AUV]{Stacks_Project}. Moreover by \cite[0EB8]{Stacks_Project} and the above exact sequence, if $\sQ^\text{univ}_q$ is torsion-free then $\sF^\text{univ}_q$ is reflexive. Since $\sQ^\text{univ}$ is proper and flat over $Q$, the set of $q\in Q$ over which $\sQ^\text{univ}_q$ is $S_1$ is open in $Q$ \cite[12.2.1]{EGA_IV.3}, call it $\sU$.
		\item \emph{There is a closed subscheme $\sL$ parametrizing those $\sF^{\text{univ}}_q$ which are closed under Lie brackets.} Consider the composite morphism
			$$L_{\sF^\text{univ}}\colon \bigwedge^2\sF^\text{univ}\to T_{X_Q/Q}\to \sQ^\text{univ}$$
	sending $v\wedge w$ to the class of $[v,w]$ in $\sQ^\text{univ}$. Its image $\im(L_{\sF^\text{univ}})$ is a coherent subsheaf of $\sQ^\text{univ}$. For a morphism $S\to Q$, the pullback morphism $(L_{\sF^\text{univ}})_S$ is zero if and only if the surjection $\im(L_{\sF^\text{univ}})\to 0$ becomes an isomorphism. By \autoref{lemma:iso_is_closed_cond} this defines a closed subscheme $\sL$ of $Q$.
		\item \emph{There is a closed subscheme $\sP$ of $\sL$ parametrizing those $\sF^{\text{univ}}_q$ which are closed under Lie brackets and under $p$-th powers.} This is similar to the previous point, with the additional subtlety that the $p$-th power is not linear. However, Hochschild's formula \autoref{eqn:Hochschild_formula} shows that $(aD)^{[p]}$ is the sum of $a^pD^{[p]}$ and a scaling of $D$; and a formula due to Jacobson \cite[(15) on p. 209]{Jacobson_Abstract_derivations_and_Lie_algebras} says that $(D_1+D_2)^{[p]}-D_1^{[p]}-D_2^{[p]}$ is a sum of commutators involving only $D_1$ and $D_2$. Thus, restricting over $\sL$, we get a linear map
				$$P_{\sF^\text{univ}}\colon F_{X_\sL/\sL}^*\left(\sF^\text{univ}_\sL\right)\longrightarrow \sQ^\text{univ}_\sL$$
		sending $v$ to the class of $v^{[p]}$ in $\sQ^\text{univ}_\sL$, where $F_{X_\sL/\sL}\colon X_\sL\to X_\sL^{(-1)}$ is the $\sL$-linear Frobenius. As above, the function $P_{\sF^\text{univ}}$ vanishes functorially over a closed subset $\sP\subset \sL$.	
	\end{enumerate}
Now we set $Q_0(r)=\sU\cap \sL$ and $Q_1(r)=\sU\cap \sP$. By construction they have the desired properties\footnote{
It might not be true that, given $S\to Q$ and the base-change $h\colon X_S\to X_Q$, the morphism $h^*(L_{\sF^\text{univ}})\colon h^*\left(\bigwedge^2 \sF^\text{univ}\right)\to h^*\sQ^\text{univ}$ is the same as $L_{h^*\sF^\text{univ}}$, for the simple reason that $h^*\left(\bigwedge^2 \sF^\text{univ}\right)$ and $\bigwedge^2h^*\sF^\text{univ}$ might be different. But in any case their images in $h^*\sQ^\text{univ}$ are the same, and all we care about is the vanishing of that image.}, and the proof is complete.
\end{proof}

\subsection{Quotients and fibers}
We turn to the following interrogation:

\begin{question}
Let $(f\colon \sX\to S,\sF\hookrightarrow T_{\sX/S})$ be a family of $1$-foliations; we continue to assume, just as in \autoref{rmk:technical_limitations}, that $T_{\sX/S}$ commutes with restriction to fibers. %over an $S_2$ base
Let $Z=\sX/\sF$ be the quotient. Since the derivations in $\sF$ are $\sO_S$-linear, the morphism $f$ factors as
		\begin{equation}\label{eqn:map_of_qut_to_base}
		\begin{tikzcd}
		\sX\arrow[rr, bend right, "f"]\arrow[r, "q"] & Z=\sX/\sF\arrow[r, "g"] & S.
		\end{tikzcd}
		\end{equation}
With these notations, we ask:
	\begin{enumerate}
		\item Is $g\colon Z\to S$ a flat morphism?
		\item For $s\in S$, how do $Z_s$ and $\sX_s/\sF_s$ relate?
	\end{enumerate}
\end{question}

For the first question, we have the following result:

\begin{proposition}\label{prop:flatness_of_quotient}
Notations as above. If $S$ is regular of dimension $\leq 2$, then $g\colon Z\to S$ is flat.
\end{proposition}
\begin{proof}
We may assume that $S=\Spec(R)$ is regular local of dimension two with maximal ideal $(x,y)$ (the dimension one case will be similar), and that $\sX=\Spec(A)$. Since $A$ is flat over $R$, the elements $x,y\in A^\sF$ form a regular sequence in $A$. By \autoref{lemma:quotient_is_normal_II} the sequence $x,y$ is also regular in $A^\sF$. By \cite[07DY]{Stacks_Project} it follows that $A^\sF$ is flat over $R$.
\end{proof}

\begin{remarks}
\begin{enumerate}
	\item The statement of \autoref{prop:flatness_of_quotient} still holds if we only assume that $\sF$ is a $1$-foliation contained in $T_{\sX/S}$, instead of being a family of $1$-foliations. Thus we generalize the flatness result of \cite[Corollary 2.8]{Schroeer_Kummer_surfaces_from_cuspidal_curves}. 
	\item \autoref{prop:flatness_of_quotient} does not hold in general if $\dim S\geq 3$. For example, consider $\sX=\bA^3_{\bold{x}}\times \bA^3_{\bold{s}}$ with $f=\pr_2\colon \sX\to \bA^3_{\bold{s}}=S$, and let $\sF$ be the $1$-foliation generated by $\partial=\sum_{i=1}^3s_i\partial_{x_i}$. This derivation defines an $\alpha_p$-action, whose fixed locus is $V(s_1,s_2,s_3)\subset \sX$. Let $\xi$ be the generic point of this fixed locus. Then $\sO_{\sX,\xi}^\partial=\sO_{Z,q(\xi)}$ has depth $2$ by \autoref{rmk:additive_quotients_are_only_S2}. Since $f(\xi)$ is the origin on $\bA^3_{\bold{s}}$, we see that the local map
			$\sO_{S,f(\xi)}\to \sO_{Z,q(\xi)}$
is not flat, for otherwise $\sO_{Z,q(\xi)}$ would have depth $3$.
\end{enumerate}
\end{remarks}

Next let us compare $Z_s$ with the quotient $\sX_s/\sF_s$. To begin with, for every $s\in S$ we construct a comparison morphism $\varphi_s\colon \sX_s/\sF_s\to Z_s$ as follows. Locally, pick a section $\partial\in \sF$: it is a local derivation on $\sX$ which, as it belongs to $T_{\sX/S}$, is $\sO_S$-linear. By the exact sequence \autoref{eqn:restriction_of_exact_seq}, the restriction $\partial_s\in \sF_s$ is a local derivation of $\sX_s$ that can be defined by the commutativity of the diagram
		$$\begin{tikzcd}
		\sO_{\sX}\arrow[r, "\partial"] \arrow[d] & \sO_{\sX}\arrow[d] \\
		\sO_{\sX_s}\arrow[r, "\partial_s"] & \sO_{\sX_s}
		\end{tikzcd}$$
So if $a\in\sO_{\sX}$ is annihilated by $\partial$, its image $\bar{a}\in \sO_{\sX_s}$ is annihilated by $\partial_s$. Letting $\partial$ range through the elements of $\sF$, we obtain a factorization of the $k(s)$-linear Frobenius of $\sX_s$ as follows:
		\begin{equation}\label{eqn:natural_morphism_quotient_fibers}
		F_{\sX_s/k(s)}^*\colon \sO_{\sX_s^{(-1)}}\to\sO_Z\otimes k(s)\overset{c_s}{\longrightarrow} \sO_{\sX_s/\sF_s}\hookrightarrow \sO_{\sX_s}.
		\end{equation}
This induces a universal homeomorphism $\varphi_s\colon \sX_s/\sF_s\to Z_s$ which factors the $k(s)$-linear Frobenius of $\sX_s$. Notice that the arrow $c_s\colon\sO_Z\otimes k(s)\to \sO_{\sX_s/\sF_s}$ in \autoref{eqn:natural_morphism_quotient_fibers} is a priori not injective. In fact, its kernel is $p$-nilpotent because $\sO_Z^p\subseteq \sO_{\sX^{(1)}}$, and as $\sO_{\sX_s}$ is reduced we get in fact that $c_s$ is an injective map if and only if $Z_s$ is reduced.

\medskip
While $\varphi_s$ is defined quite generally, we will concentrate on the case where $\sX\to S$ is smooth and $S$ is regular.

\begin{proposition}\label{prop:fiber_morphism_is_birational}
Notations as in \autoref{eqn:map_of_qut_to_base}. Assume that the exact sequence $\sF\hookrightarrow T_{\sX/S}\twoheadrightarrow \sQ$ is split at $z\in \sX$, that $f$ is smooth at $z$ and that $S$ is regular at $s=f(z)$. Then $\varphi_{s}\colon \sX_s/\sF_s\to Z_s$ is an isomorphism at $z$.
\end{proposition}
\begin{proof}
The question is local, so we may assume that $S=\Spec(A)$ and $\sX=\Spec(B)$ are affine with $(A,\fm_A)$ regular local and $B$ a smooth $A$-algebra. We may also shrink $B$ so that $\sF$ is a direct summand of $T_{B/A}$, and $T_{B/A}$ a direct summand of $T_{B/k}$. So $\sF$ is a direct summand of $T_{B/k}$. Then by %\cite[12. Theorem, p.6]{Yuan_Inseparable_Galois_theory}
\autoref{example:reg_foliations_on_reg_varieties} there is a partial system of local coordinates $v_1,\dots,v_r\in B$ such that we can write
		$$B=\bigoplus_{0\leq i_1,\dots,i_r< p}B^\sF\cdot v_1^{i_1}\cdots v_r^{i_r},$$
	and moreover $\sF$ is generated by $\partial_1,\dots,\partial_r$ with the property that $\partial_i(v_j)=\delta_{ij}$ (the Kronecker delta). By $\sO_S$-linearity, the image of the structural map $A\to B$ is contained in $B^\sF$. Thus the fiber over the closed point of $S$ is given by
			$$B/\fm_A B=\bigoplus_{0\leq i_1,\dots,i_r< p}(B^\sF/\fm_A B^\sF)\cdot \bar{v}_1^{i_1}\cdots \bar{v}_r^{i_r}$$
where $\bar{v}_i$ is the image of $v_i$ through the quotient $B\to B/\fm_A B$. Moreover $\sF_s$ is generated by the restrictions $\bar{\partial}_i$ of the $\partial_i$'s, and we still have the property $\bar{\partial}_i(\bar{v}_j)=\delta_{ij}$. Thus if $w\in B/\fm_A$ is annihilated by every element of $\sF_s$, it belongs to the summand $i_1=\dots=i_r=0$, and we can find a preimage of $w$ in the corresponding summand of $B$. That preimage will be annihilated by every element of $\sF$. This shows that the map $c_s$ in \autoref{eqn:natural_morphism_quotient_fibers} is an equality at $z\in \sX_s$. 
\end{proof}

%The hypothesis of the above proposition imply that $\sF$ is a sub-bundle of $T_{\sX/k}$, which is quite restrictive. 
Our main application of the proposition above is the following one (compare with \cite[Theorem 4.1.7]{Bergqvist_Kummer_constructions_in_families}):

\begin{corollary}\label{cor:commutativity_iff_S_2_fibers}
Notations as in \autoref{eqn:map_of_qut_to_base}. Assume that $S$ is regular. Then for every $s\in S$, the morphism $\varphi_s$ is finite and an isomorphism in codimension one. Moreover it is an isomorphism at $z\in \sX_s$ if and only if $Z_s$ is $S_2$ at $z$.
\end{corollary}
\begin{proof}
By assumption every fiber $\sX_s$ is $F$-finite, so it is easily seen from the above construction that every $\varphi_s$ is finite. Take $z\in\sX_s$ such that $\codim_{\sX_s}(z)\leq 1$: we check that the hypothesis of \autoref{prop:fiber_morphism_is_birational} hold. Since $\sX_s$ is normal, the local ring $\sO_{\sX_s,z}$ is regular. As $f$ is flat and $S$ is regular, we deduce that $\sO_{\sX,z}$ is regular. Therefore $f$ is smooth at $z\in \sX$. Moreover $\sQ_s$ is free at $z\in\sX_s$ (use the exact sequence \autoref{eqn:restriction_of_exact_seq} and the assumption that $\sF_s$ is saturated), and since $\sQ$ is flat over $S$ it follows that $\sQ$ is free at $z$ \cite[00MH]{Stacks_Project}. Thus $T_{\sX/S}\to \sQ$ splits at $z$. So \autoref{prop:fiber_morphism_is_birational} indeed applies, and $\varphi_s$ is an isomorphism at $z$. 

As $\sX_s/\sF_s$ is $S_2$, the previous paragraph implies that $\varphi_s$ is the $S_2$-fication of $\red(Z_s)$. Since $Z_s$ is generically reduced by the above paragraph, if it is $S_2$ then it is reduced everywhere and so $\varphi_s$ is an isomorphism.
\end{proof}

Another application is the following sufficient condition for commutativity (compare with \cite[Theorem 4.1.5]{Bergqvist_Kummer_constructions_in_families}):

\begin{proposition}\label{prop:commutativity_for_mu_p}
Notations as in \autoref{eqn:map_of_qut_to_base}. Assume that $S$ is regular, and that at a closed point $z\in \sX$ the $1$-foliation $\sF$ has at worst multiplicative singularities. If $s=f(z)$, then $g\colon Z\to S$ is flat at $q(z)$ and $\varphi_s$ is an isomorphism at $z\in \sX_s/\sF_s$.
\end{proposition}
\begin{proof}
This can be checked at the completions at $z$. So we may assume that $\sX=\Spec(A)$ is local complete and that $\sF$ is generated up to saturation by commuting multiplicative derivations. If $D_1,D_2\in\Der(A)$ commute, then $D_2$ descends to $\Der(A^{D_1})$; hence we may assume that $\sF$ is generated up to saturation by a single multiplicative derivation $\partial\in\Der(A)$ defining a $\mu_p$-action on $A$. Write $A=\bigoplus_{i=0}^{p-1} A_i$ where $A_i=\{a\in A\mid \partial(a)=ia\}$; each $A_i$ is an $A_0$-module, and hence an $\sO_{S,s}$-module. Then 
		$$\Tor^{\sO_{S,s}}_1(k(s),A)=\bigoplus_i \Tor^{\sO_{S,s}}_1(k(s),A_i).$$
By the local criterion for flatness \cite[00MK]{Stacks_Project}, we obtain that the flatness of $A$ over $\sO_{S,s}$ implies the flatness of $A_0=A^\partial$ over $\sO_{S,s}$.

To prove that $\varphi_s$ is an isomorphism at $z$, by \autoref{cor:commutativity_iff_S_2_fibers} we may assume that $\codim_{\sX_s}(z)\geq 2$ and show that $\depth_{q(z)}(A_0\otimes k(s))\geq 2$. (Alternatively, one can use that $A_0$ is a universal geometric quotient, see the proof of \cite[Chapter 1, \S 2, Theorem 1.1]{Mumford_Fogarty_GIT}.) Since $\sO_{S,s}$ is regular and $A_0$ is flat over it, it is equivalent to show that $\depth_{q(z)}(A_0)\geq 2+\dim f$. But $A\otimes k(s)$ is normal, so $\depth_z(A)\geq 2+\dim f$ by flatness. By (the proof of) \autoref{thm:cohom_properties_of_mult_quotient} it follows that $A_0$ has the desired property.
\end{proof}

%I claim that the fiber $Z_s$ is $S_2$ at $z$; this will conclude .  As in \autoref{eqn:splitting_square}, we extract the $(p-1)$-th root of several units in $A^\sF$ to obtain an \'{e}tale extension $A'$ of $A$ such that: the induced $1$-foliation $\sF'$ on $A'$ corresponds to a $\mu_p^{\times n}$-action given by multiplicative commuting derivations $\partial'_1,\dots,\partial_n'$, and the restrictions $\partial_i=\partial_i'|_A$ generates $\sF$ up to saturation. Now $\mu_p^{\times n}$ is linearly reductive, so taking the $\mu_p^{\times n}$-invariants, which is the same as taking the subring of $\sF'$-constants, commutes with arbitrary base-change. This follows eg from \cite[Chapter 1, \S 2, Theorem 1.1]{Mumford_Fogarty_GIT} since $\mu_p^{\times n}$-actions correspond to $(\bZ/p)^{\oplus n}$-gradings and thus have a Reynolds operator. By \autoref{cor:commutativity_iff_S_2_fibers} this implies that $(A')^{\sF'}\otimes_{\sO_{S}} k(s)$ is $S_2$. But $(A')^{\sF'}$ is an \'{e}tale extension of $A^\sF$, see \autoref{eqn:splitting_square} again, and so $(A')^{\sF'}\otimes_{\sO_S} k(s)$ is an \'{e}tale extension of $A^\sF\otimes_{\sO_{S}} k(s)$. The property $S_2$ descends \'{e}tale extensions, and thus we see that the fiber $Z_s$ is $S_2$ as desired.

We now give some examples of (non-)commutativity.

\begin{example}
Assume that $S$ is a regular curve with generic point $\eta$, that $\sX\to S$ is smooth, and let $\sF\subsetneq T_{\sX/S}$ be a family of $1$-foliations of corank $1$. If $\sQ_\eta$ is $S_2$ (equivalently a line bundle by \cite[Proposition 1.9]{Hartshorne_Stable_reflexive_sheaves}), then $\varphi_s$ is an isomorphism for every $s\in S$.

Indeed, if $s$ is a closed point then as $\sQ_s$ is $S_1$ we see that $\sQ$ is $S_2$ along $\sX_s$. Therefore $\sQ$ is an $S_2$ rank $1$ sheaf on $\sX$. As $\sX$ is regular we deduce that $\sQ$ is in fact a line bundle and we apply \autoref{prop:fiber_morphism_is_birational}.
\end{example}

\begin{example}[Quotients of $\bP^2$]
Fix an excellent DVR $R$ over $k$, with fraction field $K$. Let $\sF\subset T_{\bP^2_R/R}$ be the $1$-foliation generated by $\partial=fx\partial_x+gy\partial_y\in T_{\bA^2_{x,y}/k}$ where $f,g\in R^\times$ have distinct images in the residue field of $R$. Then $\partial$ is multiplicative. I claim that $\sF$ is lc everywhere. Indeed, cover $\bP^2_R$ with the three open sets
		$$\bA^2_{x,y} \text{ and } \bA^2_{u,v}\text{ and }\bA^2_{u',v'},\quad
		(x,y)=(u/v, 1/v)=(v'/u',1/u').$$
One has	
		$$x\partial_x=u\partial_u=-u'\partial_{u'}-v'\partial_{v'},\quad y\partial_y=-u\partial_u-v\partial_v=u'\partial_{u'}$$
Using these transformation rules and \autoref{prop:lc_foliation_and_linear_alg}, one
sees that $\sF$ is lc everywhere. Combining \autoref{thm:bir_sing_of_quotient} and \autoref{prop:commutativity_for_mu_p}, we get that $\bP^2_R/\sF\to \Spec(R)$ is a locally stable family (in the sense of \cite[\S 2.1]{Kollar_Families_of_varieties_of_general_type}) with normal special fiber.
\end{example}

\begin{example}[Quotients of $\bP^2$, continued]
In the notations of the previous example, consider on $\bA^2_{x,y}$ the $p$-closed derivation
		$$f(x^{np}-x)\partial_x+g\lambda (y^{np}-y)\partial_y$$
which on $\bA^2_{u,v}$ becomes
		$$\frac{1}{v^{np-1}}\left[
		f(u^{np-1}-v^{np-1}) 
		- g\lambda (1-v^{np-1})\right] u\partial_u
		+
		\frac{1}{v^{np-1}}\left[
		-g\lambda (1-v^{np-1})\right] v\partial_v.$$
Hence this is a multiplicative derivation, which generates an lc $1$-foliation $\sF\subset T_{\bP^2_R/R}$. Thus, as above, $\sS=\bP^2/\sF\to \Spec(R)$ is a locally stable family with normal central fiber.

Both fibers of $\sS\to \Spec(R)$ are klt surfaces of Picard rank $1$. As $K_\sF\cong \sO(np-1)$, by \autoref{prop:adjunction_formula} we see that $\sS$ is canonically polarized when $p\geq 3$ (if $p=2$ we can also get a del Pezzo or a Calabi--Yau surface). Notice that by \cite[Example 3.6]{Hirokado_Singularities_of_mult_derivations_and_Zariski_surfaces}, if $p\geq 3$ then for most choices of $\lambda$ and $n$, the minimal resolution of either fibers of $\sS$ has a non-zero non-closed global $1$-form.
\end{example}

\begin{example}\label{example:unexpected_commutativity}
Assume that $\varphi_s$ is an isomorphism: then $\sQ_s$ need not be free or even Cohen--Macaulay (which amounts to the same if $\sX_s$ is regular). Indeed, consider the constant family $\bA^3_{x,y,t}\to\bA^1_t$ over a field of characteristic $2$, and the derivation
		$$\partial=[x^2+tf]\frac{\partial}{\partial x}+[y^2+tg]\frac{\partial}{\partial y},\quad f,g\in k[t],\quad
		 \text{on }\bA^3.$$
Then $\partial\in T_{\bA^3/\bA^1}$ generates a family of $1$-foliations, say $\sF$. Let $\partial_0$ be the restriction of the derivation $\partial$ to the fiber $t=0$.
	\begin{itemize}
		\item The quotient $T_{\bA^2_{x,y}}/\sF_0$ is given by
				$$\frac{k[x,y]\frac{\partial}{\partial x}\oplus
				k[x,y]\frac{\partial}{\partial y}}{x^2\frac{\partial}{\partial x}+y^2\frac{\partial}{\partial y}},$$
		which is torsion-free but not free (and hence not Cohen--Macaulay) at the origin.
		\item However, I claim that the fiber of $k[x,y,z]^\partial$ over $t=0$ is equal to $k[x,y]^{\partial_0}$. We have computed in \autoref{example:ring_of_csts} the subring of constants of the derivation $\partial_0$: it is given by $k[x^2,y^2,xy^2-x^2y]$. Thus it suffices to find $h\in k[x,y,t]$ such that $xy^2-x^2y+th\in k[x,y,t]^\partial$. One checks that $h=gx+fy$ does the job.
	\end{itemize}
\end{example}

\begin{example}\label{example:non-commutativity}
Consider the family of $1$-foliations $\sF$ on $\sX=\bA^4_{x,y,z,t}\to \bA^1_t$ generated by the additive derivation
		$$\partial=
		x^p\frac{\partial}{\partial x}+y^p\frac{\partial}{\partial y}+t\frac{\partial}{\partial z}.$$
Then $z\in k[x,y,z]$ is a constant for $\partial_0$, so it belongs to $\sO_{\sX_0/\sF_0}$. Say that $f\in k[x,y,z,t]$ satisfies $\partial(z+tf)=0$. Then it follows that
		$$x^pf_x+y^pf_y+tf_z=-1$$
where $f_x,f_y,f_z$ are the partial derivatives of $f$. But this cannot hold in a neighbourhood of sub-scheme defined by the ideal $(x,y,t)$. Thus $\varphi_0\colon \sX_0/\sF_0\to (\sX/\sF)_0$ is not an isomorphism along the image of the line $(x=y=0)$. Observe however that it is an isomorphism everywhere else. 

While $\sX/\sF=Y\to \bA^1$ is smooth over the complement of the origin, the singularities along the central fiber are quite complicated. In fact, I claim that $Y\to S$ is not a locally stable family over $0\in \bA^1$---by which I mean that $(Y,Y_0)$ is not lc---if $p>2$. Assume it is. Then for $n>1$ not divisible by $p$, the base-change of $Y\to \bA^1_t$ along $\bA^1_u\to \bA^1_t, t=u^n$, is again locally stable over $0\in \bA^1_u$ \cite[2.16.5]{Kollar_Families_of_varieties_of_general_type}. Let us call it $Y'\to \bA^1_u$. Since $\bA^1_u\to \bA^1_t$ is flat, we can compute $Y'$ as the quotient of $\sX'=\bA^4_{x,y,z,u}$ by the $1$-foliation $\sF'$ generated by
	$$\partial'=
		x^p\frac{\partial}{\partial x}+y^p\frac{\partial}{\partial y}+u^n\frac{\partial}{\partial z}$$
(cf the proof of \autoref{lemma:localisation_properties}). Now let $\mu\colon W\to \sX'$ be the blow-up of the ideal $(x,y,z,u)$, and consider the induced commutative diagram
		$$\begin{tikzcd}
		\sX'\arrow[d] & W\arrow[l, "\mu"] \arrow[d, "q"] \\
		Y' & W/\mu^*\sF'\arrow[l]
		\end{tikzcd}$$
Then a local computation shows that the $\mu$-exceptional divisor $E$ is $\sF'$-invariant, and that
	$$K_{\mu^*\sF'}=\mu^*K_{\sF'}-\min\{n-1,p-1\}E.$$
If $F=q(E)$, then \autoref{eqn:discrepancies_of_quotient} shows that
		$$a(Y';F)=3-(p-1)\cdot \min\{n-1,p-1\}.$$
This is smaller than $-1$ as soon as $p,n>2$, but contradicts the log canonicity of $(Y',Y'_0)$.
\end{example}

\bibliographystyle{alpha}
\bibliography{Bibliography}

\end{document}